\documentclass[english]{article}
\usepackage[T1]{fontenc}
\usepackage[latin9]{inputenc}
\usepackage{array}
\usepackage{float}
\usepackage{multirow}
\usepackage{amsthm}
\usepackage{amsmath}
\usepackage{amssymb}
\usepackage{graphicx}

\makeatletter

\providecommand{\tabularnewline}{\\}
\floatstyle{ruled}
\newfloat{algorithm}{tbp}{loa}
\providecommand{\algorithmname}{Algorithm}
\floatname{algorithm}{\protect\algorithmname}

\theoremstyle{plain}
\newtheorem{thm}{\protect\theoremname}
  \theoremstyle{plain}
  \newtheorem{lem}[thm]{\protect\lemmaname}
  \theoremstyle{plain}
  \newtheorem{prop}[thm]{\protect\propositionname}

\usepackage{algorithmic}
\def\vec#1{\mbox{\boldmath $#1$}}

\makeatother

\usepackage{babel}
  \providecommand{\lemmaname}{Lemma}
  \providecommand{\propositionname}{Proposition}
\providecommand{\theoremname}{Theorem}

\begin{document}

\title{OPINS: An Orthogonally Projected Implicit Null-Space Method for Singular
and Nonsingular Saddle-Point Systems\footnotemark[1]}

\author{Cao Lu\footnotemark[2]\and Tristan Delaney\footnotemark[2] \and
Xiangmin Jiao\footnotemark[2]\ \footnotemark[3]}

\maketitle
\footnotetext[1]{This work was supported by DoD-ARO under contract \#W911NF0910306. The third author is also supported by a subcontract to Stony Brook University from Argonne National Laboratory under Contract DE-AC02-06CH11357 for the SciDAC program funded by the Office of Science, Advanced Scientific Computing Research of the U.S. Department of Energy.}
\footnotetext[2]{Dept. of Applied Math. \& Stat., Stony Brook University, Stony Brook, NY 11794, USA.}
\footnotetext[3]{Corresponding author. Email: jiao@ams.sunysb.edu.}
\begin{abstract}
Saddle-point systems appear in many scientific and engineering applications.
The systems can be sparse, symmetric or nonsymmetric, and possibly
singular. In many of these applications, the number of constraints
is relatively small compared to the number of unknowns. The traditional
null-space method is inefficient for these systems, because it is
expensive to find the null space explicitly. Some alternatives, notably
constraint-preconditioned or projected Krylov methods, are relatively
efficient, but they can suffer from numerical instability and even
nonconvergence. In addition, most existing methods require the system
to be nonsingular or be reducible to a nonsingular system. In this
paper, we propose a new method, called \emph{OPINS}, for singular
and nonsingular saddle-point systems. OPINS is equivalent to the null-space
method with an orthogonal projector, without forming the orthogonal
basis of the null space explicitly. OPINS can not only solve for the
unique solution for nonsingular saddle-point problems, but also find
the minimum-norm solution in terms of the solution variables for singular
systems. The method is efficient and easy to implement using existing
Krylov solvers for singular systems. At the same time, it is more
stable than the other alternatives, such as projected Krylov methods.
We present some preconditioners to accelerate the convergence of OPINS
for nonsingular systems, and compare OPINS against some present state-of-the-art
methods for various types of singular and nonsingular systems.
\end{abstract}
\pagestyle{myheadings}
\thispagestyle{plain}
\markboth{C. LU, T. DELANEY, AND X. JIAO}{ORTHOGONALLY PROJECTED IMPLICIT NULL-SPACE METHOD}

\section{Introduction}

We consider the numerical solution of the following saddle-point system
\begin{equation}
\begin{bmatrix}\vec{A} & \vec{B}^{T}\\
\vec{B} & \vec{0}
\end{bmatrix}\begin{bmatrix}\vec{x}\\
\vec{y}
\end{bmatrix}=\begin{bmatrix}\vec{f}\\
\vec{g}
\end{bmatrix},\label{eq:KKT}
\end{equation}
where $\vec{A}\in\mathbb{R}^{n\times n}$ may be nonsymmetric, $\vec{B}\in\mathbb{R}^{m\times n}$,
\emph{$\vec{x},\vec{f}\in\mathbb{R}^{n}$,} and $\vec{y},\vec{g}\in\mathbb{R}^{m}$.
The coefficient matrix, denoted by $\vec{K}$, is in general nonsymmetric
and indefinite. It may be nonsingular or singular, and often very
ill-conditioned even when it is nonsingular. 

This type of saddle-point problems arises in many scientific applications.
For example, it can arise from solving the following constrained optimization
problem,
\begin{equation}
\min_{\vec{x}}\frac{1}{2}\vec{x}^{T}\vec{A}\vec{x}-\vec{f}^{T}\vec{x}\quad\mbox{subject to}\quad\vec{B}\vec{x}=\vec{g},\label{eq:LangrangeMultiplier}
\end{equation}
using the method of Lagrange multipliers, where $\vec{A}$ is the
Hessian of the quadratic objective function, and $\vec{B}\vec{x}=\vec{g}$
defines a constraint hyperplane for the minimization. In this case,
$\vec{A}$ is typically symmetric, $\vec{x}$ contains the \emph{optimization
variables} or \emph{solution variables}, and $\vec{y}$ contains the
\emph{Lagrange multipliers} or \emph{constraint variables}. The optimality
conditions are referred as Karush-Kuhn-Tucker conditions, and the
system (\ref{eq:KKT}) is often called the \emph{KKT system} \cite{nocedal2006numerical},
with $\vec{K}$ being the\emph{ KKT matrix}. In PDE discretization,
the Lagrange multipliers are often used to enforce nodal conditions,
such as sliding boundary conditions and continuity constraints at
hanging nodes \cite{fries2011hanging}. 

Solving the saddle-point problems is particularly challenging. One
may attempt to solve it using a general-purpose solver, such as a
direct solver or an iterative solver \cite{bramble1988preconditioning,de2005block,tuma2002note},
or the range-space method \cite{vavasis1994stable}. For large and
sparse saddle-point systems, a more powerful method is the null-space
method \cite{ANU:298722}, which solves for $\vec{x}$ first and then
$\vec{y}$, with an iterative method as its core solver. The null-space
method computes $\vec{x}$ by first finding a particular solution,
and then solving its complementary component in $\mbox{null}(\vec{B})$
(c.f. Algorithm~\ref{alg:NULLSPACE}). This method is physically
and geometrically meaningful, as it effectively finds a solution $\vec{x}_{*}$
within the constraint hyperplane such that $\vec{A}\vec{x}_{*}-\vec{f}\in\mbox{range}(\vec{B}^{T})$.
This is satisfied if and only if $\vec{f}\in\mbox{range}(\vec{A})+\mbox{range}(\vec{B}^{T})$.
Let $\vec{Z}$ denote the matrix composed of a set of basis vectors
of $\mbox{null}\left(\vec{B}\right)$, and we say $\vec{Z}$ is \emph{orthonormal
}if the basis is orthonormal. The traditional null-space method is
efficient if $\mbox{null}(\vec{B})$ is low dimensional, since it
involves a small linear system with the coefficient matrix $\vec{Z}^{T}\vec{A}\vec{Z}$.
In exact arithmetic, $\vec{Z}$ does not need to be orthonormal. However,
with rounding errors, it is desirable for $\vec{Z}$ to be orthonormal
(or nearly orthonormal), so that the condition number of $\vec{Z}^{T}\vec{A}\vec{Z}$
is close to that of $\vec{A}$. However, if the dimension of $\mbox{null}(\vec{B})$
is high, i.e., when $m\ll n$, determining such a nearly orthonormal
basis is computationally expensive, making the traditional null-space
method impractical for such applications.

Another class of methods, which we refer to as \emph{implicit null-space
methods}, are effective for large-scale saddle-point systems. Examples
of this method include the Krylov subspace methods with a constraint
preconditioner \cite{schoberl2007symmetric,gould2001solution} and
the projected Krylov methods (KSP) \cite{gould2014projected}. In
exact arithmetic, these methods are equivalent to the null-space method
with an orthonormal basis of $\mbox{null}(\vec{B})$, but they do
not require computing the basis explicitly. These methods are usually
more efficient than directly applying the Krylov subspace methods
to the whole system. However, with rounding errors, these methods
may suffer from numerical instability, due to loss of orthogonality.
The instability may be mitigated by iterative refinements \cite{gould2001solution,gould2014projected},
but there is no guarantee that the solver would not stagnate even
with iterative refinements.

In this paper, we propose an implicit orthogonal projection method,
called \emph{OPINS}, which is a more stable variant of implicit null-space
methods. Specifically, we compute an orthonormal basis for $\mbox{range}\left(\vec{B}^{T}\right)$,
which can be constructed efficiently when $m\ll n$ using a stable
algorithm such as $QR$ factorization with column pivoting (QRCP).
Let $\vec{U}$ denote the matrix composed of such an orthonormal basis.
Instead of using $\vec{Z}^{T}\vec{A}\vec{Z}$ in the null-space method,
we use the orthogonal projector $\vec{\Pi}_{\vec{U}}^{\perp}\equiv\vec{I}-\vec{U}\vec{U}^{T}$
to construct a singular but compatible linear system, with a symmetric
coefficient matrix $\vec{\Pi}_{\vec{U}}^{\perp}\vec{A}\vec{\Pi}_{\vec{U}}^{\perp}$.
We solve this system using a solver for singular systems, such as
MINRES \cite{paige1975solution,choi2011minres} and SYMMLQ \cite{paige1975solution}
for symmetric systems, and GMRES \cite{Saad86GMRES,Reichel2005BGS},
LSQR \cite{Paige92LSQR} or LSMR \cite{Fong11LSMR} for nonsymmetric
systems. We also propose preconditioners for saddle-point systems,
based on the work in \cite{gould2014projected}. The resulting OPINS
method is highly efficient when $m\ll n$, and it is stable, robust,
and easy to implement using existing Krylov-subspace method for singular
systems, without the issue of stagnation suffered by other implicit
null-space methods.

Besides being more stable, another advantage of OPINS is that it can
be applied to singular saddle-point systems, which arise in various
applications. For example, in the application of fracture mechanics,
a solid object may contain many cracks and isolated pieces, and an
elasticity model is in general under-constrained, leading to singular
saddle-point systems. In the literature, most methods assume the saddle-point
system is nonsingular. When this assumption is violated, for example
when there are redundant constraints or insufficient constraints,
these methods may force the user to add artificial boundary conditions
or soft constraints to make the system nonsingular, which unfortunately
may undermine the physical accuracy of the solution. We consider singular
saddle-point systems under the assumption that $\vec{f}\in\mbox{range}(\vec{A})+\mbox{range}(\vec{B}^{T})$,
in other words, the saddle-point system is compatible in terms of
$\vec{f}$. For problems arising from constrained minimization, such
a system finds the solution that minimizes $\Vert\vec{x}\Vert$ among
all the solutions of 
\begin{equation}
\min_{\vec{x}}\frac{1}{2}\vec{x}^{T}\vec{A}\vec{x}-\vec{f}^{T}\vec{x}\quad\mbox{subject to }\quad\min_{\vec{x}}\mbox{\ensuremath{\Vert}}\vec{g}-\vec{B}\vec{x}\Vert.\label{eq:SingularLangrangeMultiplier}
\end{equation}
Throughout this paper, the norms are in 2-norm unless otherwise noted.
We refer to the above solution as the \emph{minimum-norm solution}
of a singular saddle-point system. Geometrically, this effectively
defines a constraint hyperplane in a least-squares sense, and then
require the objective function to be minimized exactly within the
hyperplane, and it defines a unique solution under the assumption
of $\vec{f}\in\mbox{range}(\vec{A})+\mbox{range}(\vec{B}^{T})$. Conversely,
if $\vec{f}\not\in\mbox{range}(\vec{A})+\mbox{range}(\vec{B}^{T})$,
the objective function does not have a critical point in the constraint
hyperplane, so the constrained minimization is ill-posed. OPINS delivers
a more accurate and systematic method to find the minimum-norm solution
for such singular saddle-point problems.

The remainder of the paper is organized as follows. Section~\ref{sec:Background}
reviews some background and related work. Section~\ref{sec:OPINS}
introduces the Orthogonally Projected Implicit Null-Space Method (OPINS)
for singular and nonsingular saddle-point systems. Section~\ref{sec:Preconditioners}
introduces effective preconditioners for OPINS. Section~\ref{sec:Results}
presents some numerical results with OPINS as well as comparisons
against some state-of-the-art methods. Section~\ref{sec:Conclusions}
concludes the paper with a discussion on future research directions.

\section{\label{sec:Background}Background}

We briefly review some important concepts and related methods for
solving saddle-point problems. There is a large body of literature
on saddle-point problems; see \cite{ANU:298722} for a comprehensive
survey up to early 2000s and the references in \cite{gould2014projected}
for more recent works. We focus our discussions on the null-space
methods, since they are the most relevant to our proposed method.
For completeness, we will also briefly discuss some other methods
for saddle-point systems and singular linear systems.

\subsection{Explicit Null-Space Methods}

One of the most powerful methods for solving saddle-point systems
is the null-space method. Geometrically, for a nonsingular saddle-point
problem arising from constrained minimization, this method finds a
critical point of an objective function within a constraint hyperplane
$\vec{B}\vec{x}=\vec{g}$. Algorithm~\ref{alg:NULLSPACE} outlines
the algorithm, where $\vec{x}_{p}\in\mathbb{R}^{n}$ denotes a particular
solution within the constraint hyperplane. Let $q$ denote the rank
of the constraint matrix $\vec{B}$, and $q\leq m$. The column vectors
of the $\vec{Z}\in\mathbb{R}^{n\times(n-q)}$ form a basis of $\mbox{null}(\vec{B})$,
and $\vec{B}\vec{Z}=\vec{0}$. Step 3 finds a component $\vec{x}_{n}$
in $\mbox{null}(\vec{B})$, i.e. a vector tangent to the constraint
hyperplane, so that $\vec{x}_{n}+\vec{x}_{p}$ is at a critical point.
After determining $\vec{x}$, the final step finds the Lagrange multipliers
$\vec{y}$ in $\mbox{range}(\vec{B})$. 

\begin{algorithm}
\protect\caption{\label{alg:NULLSPACE}Null-Space Method}

\textbf{input}: $\vec{A}$, $\vec{B}$, $\vec{f}$, $\vec{g}$, tolerance
for the iterative solver (if used)

\textbf{output}: $\vec{x}$, $\vec{y}$ (optional)

\begin{algorithmic}[1]

\STATE solve $\vec{B}\vec{x}_{p}=\vec{g}$

\STATE compute $\vec{Z}$, composed of basis vectors of $\mbox{null}(\vec{B})$

\STATE \label{-solve null-space}solve $\vec{Z}^{T}\vec{A}\vec{Z}\vec{v}=\vec{Z}^{T}\left(\vec{f}-\vec{A}\vec{x}_{p}\right)$

\STATE $\vec{x}_{n}\leftarrow\vec{Z}\vec{v}$ and $\vec{x}\leftarrow\vec{x}_{p}+\vec{x}_{n}$

\STATE $\vec{y}\leftarrow\vec{B}^{+T}(\vec{f}-\vec{A}\vec{x})$

\end{algorithmic}
\end{algorithm}

In the algorithm, step 3 is the most critical, which solves the equation
\begin{equation}
\vec{Z}^{T}\vec{A}\vec{Z}\vec{v}=\vec{Z}^{T}\left(\vec{f}-\vec{A}\vec{x}_{p}\right),\label{eq:NullSpaceEq}
\end{equation}
within $\mbox{null}(\vec{B})$. We refer to (\ref{eq:NullSpaceEq})
as the \emph{null-space equation}, and denote its coefficient matrix
by $\hat{\vec{N}}$. This matrix is $(n-q)\times(n-q)$, which is
smaller than the original matrix $\vec{K}$. When $\vec{A}$ is symmetric
and positive semidefinite and $\mbox{null}\left(\vec{A}\right)\cap\mbox{null}\left(\vec{B}\right)=\left\{ \vec{0}\right\} $,
$\hat{\vec{N}}$ is SPD \cite{ANU:298722}, and (\ref{eq:NullSpaceEq})
can be solved efficiently using preconditioned conjugate gradient
(CG). However, if $\hat{\vec{N}}$ is symmetric but indefinite or
is nonsymmetric, then an alternative iterative solver, such MINRES,
SYMMLQ \cite{paige1975solution}, or GMRES \cite{Saad86GMRES}, can
be used; see textbooks such as \cite{Saad03IMS} for details of these
iterative solvers.

In the traditional null-space method, the matrix $\vec{Z}$ is constructed
explicitly. We refer to such an approach as the \emph{explicit null-space
method}. In exact arithmetic, it is not necessary for $\vec{Z}$ to
be orthonormal. If $\vec{B}$ has full rank, then there exists an
$n\times n$ permutation matrix $\vec{P}$ such as $\vec{B}\vec{P}=\left[\vec{B}_{1}\mid\vec{B}_{2}\right]$,
where $\vec{B}_{1}$ is an $m\times m$ nonsingular matrix, and $\vec{B}_{2}$
is $m\times(n-m)$. Then, one could simply choose $\vec{Z}$ to be
\cite{gilbert1987computing} 
\[
\vec{Z}=\vec{P}\left[\begin{array}{c}
-\vec{B}_{1}^{-1}\vec{B}_{2}\\
\vec{I}
\end{array}\right],
\]
where $\vec{I}$ is the $(n-m)\times(n-m)$ identity matrix. However,
if the columns of $\vec{Z}$ are too far from being orthonormal, $\vec{Z}^{T}\vec{A}\vec{Z}$
may be ill-conditioned, which in turn can cause slow convergence for
iterative solvers or large errors in the resulting solution. Therefore,
it is desirable for $\vec{Z}$ to be orthonormal or nearly orthonormal,
so that $\vec{Z}^{T}\vec{A}\vec{Z}$ (approximately) preserves the
condition number of $\vec{A}$. If the dimension of $\mbox{null}(\vec{B})$
is high, i.e., when $m\ll n$, $\vec{Z}$ is typically quite dense.
Determining an orthonormal $\vec{Z}$ requires the full QR factorization
of $\vec{B}^{T}$, which takes $\mathcal{O}(n^{3})$ operations. Therefore,
explicit null-space methods are impractical for large-scale applications.

\subsection{Implicit Null-Space Methods}

For saddle-point systems where $m\ll n$, it is desirable to avoid
constructing an orthonormal basis of $\mbox{null}(\vec{B}$) explicitly.
One such approach is to use a Krylov subspace method with a \emph{constraint
preconditioner} \cite{gould2001solution,keller2000constraint}, which
has the form
\begin{equation}
\vec{M}=\left[\begin{array}{cc}
\vec{G} & \vec{B}^{T}\\
\vec{B} & \vec{0}
\end{array}\right],\label{eq:constraint_prec}
\end{equation}
where $\vec{G}$ is an approximation of $\vec{A}$. This preconditioner
is indefinite, and it was shown in \cite{schoberl2007symmetric} that
it provides optimal bounds for the maximum eigenvalues among similar
indefinite preconditioners. Since the preconditioner is indefinite,
it is not obvious whether we can use it as a preconditioner for methods
such as CG or MINRES, which typically require symmetric positive-definite
preconditioners. As shown in \cite{gould2001solution,gould2014projected},
one can apply the preconditioner to the modified saddle-point system
\begin{equation}
\begin{bmatrix}\vec{A} & \vec{B}^{T}\\
\vec{B} & \vec{0}
\end{bmatrix}\begin{bmatrix}\vec{x}_{n}\\
\vec{y}
\end{bmatrix}=\begin{bmatrix}\vec{f}-\vec{A}\vec{x}_{p}\\
\vec{0}
\end{bmatrix},\label{eq:PMINRES}
\end{equation}
with preconditioned CG and MINRES. The vector $\vec{x}_{n}$ in (\ref{eq:PMINRES})
is the same as that from step 4 in Algorithm~\ref{alg:NULLSPACE},
but the vector $\vec{y}$ computed from (\ref{eq:PMINRES}) may be
inaccurate, so one may need to solve for $\vec{y}$ separately once
$\vec{x}$ is obtained. In exact arithmetic, the Krylov subspace method
with this constraint preconditioner is equivalent to the \emph{projected
Krylov method} in \cite{gould2014projected}. Because of this equivalence,
we will use the names \emph{projected Krylov methods} and \emph{Krylov
subspace methods with constraint preconditioning} interchangeably
in this paper, although these methods somewhat differ in their implementation
details.

The projected Krylov method is closely related to the preconditioned
null-space method on the null-space equation (\ref{eq:NullSpaceEq})
with an orthonormal $\vec{Z}$. The stability of the method requires
the computed $\vec{x}_{n}$ to be exactly in $\mbox{null}(\vec{B})$
at each step. However, the rounding errors can quickly introduce a
nonnegligible component in $\mbox{range}(\vec{B}^{T})$, which can
cause the method to break down \cite{gould2001solution,gould2014projected}.
To mitigate this issue, the constraint preconditioner must be applied
``exactly,'' by solving the preconditioned system with a direct
method followed by one or more steps of iterative refinement per iteration
\cite{gould2001solution,gould2014projected}. The iterative refinement
introduces extra cost, and there is no guarantee that it would recover
orthogonality between $\vec{x}_{n}$ and $\mbox{range}(\vec{B})$
to machine precision, so the projected Krylov method may still stagnate.
Besides its potential instability, the projected Krylov methods typically
assume the KKT system is nonsingular \cite{gould2014projected}. It
is desirable to develop a more stable version of the implicit null-space
method that can also be applied to singular KKT systems. The OPINS
method proposed in this paper achieves this goal.

\subsection{Other Methods for Saddle-Point Systems}

Besides the null-space methods, another class of methods for saddle-point
problems is the range-space method \cite{vavasis1994stable}. It first
obtains $\vec{y}$ by solving the system
\begin{equation}
\left(\vec{B}\vec{A}^{-1}\vec{B}^{T}\right)\vec{y}=\vec{B}\vec{A}^{-1}\vec{f}-\vec{g},\label{eq:RangeSpace1}
\end{equation}
where the coefficient matrix is the Schur complement, and then computes
$\vec{x}$ by solving
\begin{equation}
\vec{A}\vec{x}=\vec{f}-\vec{B}^{T}\vec{y}.\label{eq:RangeSpace2}
\end{equation}
The range-space method can be attractive if a factorization of $\vec{A}$
is available. However, computing the Schur complement is expensive
if $\vec{A}$ is large and sparse, and the method is not directly
applicable if $\vec{A}$ is singular.

The null-space and range-space methods both leverage the special structures
of the saddle-point systems. In some cases, one may attempt to solve
the whole system (\ref{eq:KKT}) directly using a factorization-based
method, such as $LDL^{T}$ decomposition for symmetric systems \cite{tuma2002note}.
These methods are prohibitively expensive for large-scale problems.
In addition, since the solution variables $\vec{x}$ and constraint
variables $\vec{y}$ have different physical meanings, the entries
in $\vec{A}$ and $\vec{B}$ may have very different scales. As a
result, $\vec{K}$ may be arbitrarily ill-conditioned, and these methods
may break down or produce inaccurate solutions due to poor scaling.
Another class of method is iterative solvers with preconditioners
\cite{Benzi08SPT}. The constraint preconditioner is a special case,
which is equivalent to an implicit null-space method, with some optimality
properties among similar indefinite preconditioners \cite{bank1989class,schoberl2007symmetric}.
Some other preconditioners include the block diagonal \cite{de2005block},
block triangular \cite{bramble1988preconditioning}, and multigrid
\cite{adams2004algebraic}. These methods do not distinguish between
$\vec{x}$ and $\vec{y}$, so they tend to have more difficulties
when $\vec{K}$ is singular or ill-conditioned.

\subsection{Methods for Singular Saddle-Point Systems}

The methods we discussed above typically assume nonsingular saddle-point
systems. For general singular saddle-point problems, one may resort
to truncated SVD \cite{Golub13MC} or rank-revealing QR, which are
computationally expensive. One may also apply an iterative solver
for singular systems, such as MINRES and SYMMLQ \cite{paige1975solution}
for compatible symmetric systems, MINRES-QLP \cite{choi2011minres}
for incompatible symmetric systems, Breakdown-free GMRES \cite{Reichel2005BGS},
LSQR \cite{Paige92LSQR} and LSMR \cite{Fong11LSMR} for nonsymmetric
systems, and GMRES \cite{Saad86GMRES} for compatible nonsymmetric
systems with $\mbox{range}(\vec{K})=\mbox{range}(\vec{K}^{T})$. Without
preconditioners, except for Breakdown-free GMRES \cite{Reichel2005BGS},
these methods can find the minimum-norm solution of compatible singular
systems when they are applicable. However, these methods do not distinguish
between $\vec{x}$ and $\vec{y}$ in the solver. As a result, these
methods minimize $\Vert\vec{x}\Vert^{2}+\Vert\vec{y}\Vert^{2}$, which
may substantially differ from the minimum-norm solution of $\vec{x}$
as defined in (\ref{eq:SingularLangrangeMultiplier}), especially
when $\vec{K}$ is ill-conditioned. OPINS will overcome this issue
by leveraging the iterative solvers in the implicit null-space method
in a stable and efficient fashion, and in turn offer a new effective
method for solving singular saddle-point systems.

\section{\label{sec:OPINS}Proposed Method}

In this section, we introduce the \emph{Orthogonally Projected Implicit
Null-Space Method}, or \emph{OPINS}, for solving saddle-point systems
(\ref{eq:KKT}). Similar to the other implicit null-space methods,
OPINS does not require the explicit construction of the basis of $\mbox{null}(\vec{B})$,
and it is particularly effective when $m\ll n$. However, unlike previous
null-space methods, OPINS enforces orthogonality explicitly and hence
enjoys better stability. It is also applicable to singular saddle-point
systems. In the following subsections, we present the OPINS method,
its derivation, and the analysis of its cost and stability.

\subsection{Algorithm Description}

A core idea of OPINS is to use the orthogonal projector onto $\mbox{null}(\vec{B})$,
constructed from an orthonormal basis of $\mbox{range}(\vec{B}^{T})$.
Let $\vec{U}$ denote the matrix composed of an orthonormal basis
of $\mbox{range}(\vec{B}^{T})$, which can be computed using truncated
SVD or QR with column pivoting (QRCP) \cite{chan1987rank,Golub13MC}.
Since $\vec{U}$ is orthonormal, $\vec{\Pi}_{\vec{U}}\equiv\vec{U}\vec{U}^{T}$
is the unique orthogonal projector onto $\mbox{range}(\vec{B}^{T})$,
and 
\[
\vec{\Pi}_{\vec{U}}^{\perp}\equiv\vec{I}-\vec{\Pi}_{\vec{U}}=\vec{I}-\vec{U}\vec{U}^{T}
\]
is its complementary orthogonal projector onto $\mbox{null}(\vec{B})$.
Note that $\vec{\Pi}_{\vec{U}}^{\perp}=\vec{\Pi}_{\vec{Z}}$, where
$\vec{Z}$ is composed of an orthonormal basis of $\mbox{null}(\vec{B})$.
Let $q=\mbox{rank}(\vec{B})$. Using this projector, OPINS replaces
the $(n-q)\times(n-q)$ null-space equation (\ref{eq:NullSpaceEq})
in the null space method with the $n\times n$ singular system 
\begin{equation}
\vec{\Pi}_{\vec{U}}^{\perp}\vec{A}\vec{\Pi}_{\vec{U}}^{\perp}\vec{w}=\vec{\Pi}_{\vec{U}}^{\perp}\left(\vec{f}-\vec{A}\vec{x}_{p}\right),\label{eq:OPINSNullSpaceEq}
\end{equation}
and then solves it using a solver for singular systems. We refer to
the above equation as the \emph{projected null-space (PNS) equation},
and denote its coefficient matrix as $\vec{N}$.

Algorithm~\ref{alg:OPINS} outlines the complete OPINS algorithm,
which applies to nonsingular or compatible singular systems. The first
two steps find an orthonormal basis of $\vec{B}^{T}$ using QRCP,
where $\vec{P}$ is an $m\times m$ permutation matrix, so that the
diagonal values of $\vec{R}$ are sorted in descending order. $\vec{Q}\in\mathbb{R}^{n\times m}$
is orthonormal, and $\vec{R}\in\mathbb{R}^{m\times m}$ is upper triangular.
For stability, QRCP should be computed based on Householder QR factorization
\cite{Golub13MC}. If $\vec{B}$ is rank deficient, its rank can be
estimated from the magnitude of the diagonal entries in $\vec{R}$,
or more robustly using a condition-number estimator \cite{Golub13MC}.
The first $q$ columns of $\vec{Q}$ form an orthonormal basis of
$\mbox{range}\left(\vec{B}^{T}\right)$. With QRCP, both $\vec{x}_{p}$
in step 3 and $\vec{y}$ in step 6 can also be solved efficiently.
When $\vec{B}$ is rank deficient, so is $\vec{R}$. We use $\vec{R}_{1:q,1:q}^{-T}$
and $\vec{R}_{1:q,1:q}^{-1}$ to denote the forward and back substitutions
on $\vec{R}_{1:q,1:q}$. 

\begin{algorithm}
\protect\caption{\label{alg:OPINS}OPINS: Orthogonally Projected Implicit Null-Space
Method}

\textbf{input}: $\vec{A}$, $\vec{B}$, $\vec{f}$, $\vec{g}$, tolerances
for rank estimation and iterative solver

\textbf{output}: $\vec{x}$, $\vec{y}$ (optional)

\begin{algorithmic}[1]

\STATE $\vec{B}^{T}\vec{P}=\vec{Q}\vec{R}$ \COMMENT{QR factorization
with column pivoting}

\STATE $\vec{U}\leftarrow\vec{Q}_{1:q}$, where $q$ is $\mbox{rank}(\vec{B})$
estimated from QRCP

\STATE $\vec{x}_{p}\leftarrow\vec{B}^{+}\vec{g}=\vec{U}\vec{R}_{1:q,1:q}^{-T}\left(\vec{P}^{T}\vec{g}\right)_{1:q}$

\STATE solve $\vec{\Pi}_{\vec{U}}^{\perp}\vec{A}\vec{\Pi}_{\vec{U}}^{\perp}\vec{w}=\vec{\Pi}_{\vec{U}}^{\perp}\left(\vec{f}-\vec{A}\vec{x}_{p}\right)$
using iterative singular solver

\STATE $\vec{x}\leftarrow\vec{x}_{p}+\vec{x}_{n}$, where $\vec{x}_{n}=\vec{\Pi}_{\vec{U}}^{\perp}\vec{w}$

\STATE $\vec{y}\leftarrow\vec{B}^{+T}\left(\vec{f}-\vec{A}\vec{x}\right)=\vec{P}_{:,1:q}\vec{R}_{1:q,1:q}^{-1}\vec{U}^{T}\left(\vec{f}-\vec{A}\vec{x}\right)$

\end{algorithmic}
\end{algorithm}

The key step of the algorithm is the solution of the PNS equation
(\ref{eq:OPINSNullSpaceEq}) in step 4, which is singular. As we will
show in Section~\ref{sub:derivation}, (\ref{eq:OPINSNullSpaceEq})
is compatible under the assumption that (\ref{eq:KKT}) is compatible
in terms of $\vec{f}$, i.e., $\vec{f}\in\mbox{range}(\vec{A})+\mbox{range}(\vec{B}^{T})$.
Therefore, we can solve it using a Krylov subspace method for compatible
singular systems, as we discuss in more detail in Subsection~\ref{sub:solution-OPNS}.
A key operation in these methods is the multiplication of the coefficient
matrix with a vector. For the multiplication with $\vec{\Pi}_{\vec{U}}^{\perp}$
with any vector $\vec{v}\in\mathbb{R}^{n}$, 
\[
\vec{\Pi}_{\vec{U}}^{\perp}\vec{v}=\vec{v}-\vec{U}\left(\vec{U}^{T}\vec{v}\right),
\]
which can be computed stably and efficiently. Note that $\vec{U}$
is not stored explicitly either, but as a collection of Householder
reflection vectors in QRCP.

\subsection{\label{sub:derivation}Detailed Derivation of OPINS}

The derivation of OPINS is similar to that of the null-space method.
For completeness, we will start with the derivation of the explicit
null-space method, and then extend it to derive OPINS. We will also
discuss the solution techniques of the PNS equation.

\subsubsection{Null-Space Method for Singular and Nonsingular Systems}

The null-space method is typically derived algebraically for nonsingular
saddle-point systems. Since we also consider singular systems that
may be partially incompatible, it is instructive to consider an alternative
but equivalent derivation, which has clear geometric meanings for
singular saddle-point problems in the context of constraint minimization
(\ref{eq:SingularLangrangeMultiplier}). 

Suppose (\ref{eq:KKT}) is a saddle-point system compatible in terms
of $\vec{f}$, i.e., $\vec{f}\in\mbox{range}(\vec{A})+\mbox{range}(\vec{B}^{T})$.
Then, there exists $\vec{x}_{*}\in\mathbb{R}^{n}$ and $\vec{y}_{*}\in\mathbb{R}^{m}$,
such that 
\begin{equation}
\vec{A}\vec{x}_{*}-\vec{f}=-\vec{B}^{T}\vec{y}_{*}\mbox{ subject to }\min_{\vec{x}_{*}}\Vert\vec{g}-\vec{B}\vec{x}_{*}\Vert.\label{eq:KKT_general}
\end{equation}
The constraint defines a hyperplane, in the least squares sense, whose
tangent space is orthogonal to $\vec{B}^{T}$. The vector $\vec{Ax}_{*}-\vec{f}$
corresponds to the gradient of a nonlinear objective function. Equation~(\ref{eq:KKT_general})
indicates that $\vec{A}\vec{x}_{*}-\vec{f}$ is a linear combination
of the column vectors of $\vec{B}^{T}$ with the Lagrange multipliers
$\vec{y}_{*}$ as the coefficients. In other words, $\vec{f}-\vec{A}\vec{x}_{*}$
is orthogonal to the constraint hyperplane. Hence, we can rewrite
(\ref{eq:KKT_general}) as $\vec{Z}^{T}(\vec{A}\vec{x}_{*}-\vec{f})=\vec{0}$,
where $\vec{Z}$ is a basis of $\mbox{null}(\vec{B})$. The saddle-point
system is then equivalent to 
\[
\vec{Z}^{T}\vec{A}\vec{x}_{*}=\vec{Z}^{T}\vec{f}\mbox{ subject to }\min_{\vec{x}_{*}}\Vert\vec{g}-\vec{B}\vec{x}_{*}\Vert,
\]
which is an alternative statement of the constraint minimization.

Let $\vec{x}_{*}=\vec{x}_{n}+\vec{x}_{p}$, where $\vec{x}_{p}\in\mbox{range}(\vec{B}^{T})$
is a point on the constraint hyperplane. There exists $\vec{x}_{n}=\vec{Z}\vec{v}\in\mbox{null}(\vec{B})$,
i.e. a vector tangent to the constraint hyperplane, such that 
\[
\vec{Z}^{T}\vec{A}(\vec{x}_{n}+\vec{x}_{p})=\vec{Z}^{T}\vec{f},
\]
or equivalently, 
\[
\vec{Z}^{T}\vec{A}\vec{Z}\vec{v}=\vec{Z}^{T}\left(\vec{f}-\vec{A}\vec{x_{p}}\right).
\]
This is the null-space equation (\ref{eq:NullSpaceEq}). If $\vec{f}\in\mbox{range}(\vec{A})+\mbox{range}(\vec{B}^{T})$,
then the null-space equation (\ref{eq:NullSpaceEq}) is compatible.

In the above derivation, $\vec{Z}$ does not need to be orthonormal.
However, if $\vec{Z}$ is far from being orthonormal, the system (\ref{eq:NullSpaceEq})
may have a large condition number, which may affect the convergence
of the iterative solvers and the accuracy of the numerical solution.

\subsubsection{OPINS for Singular and Nonsingular Systems}

To derive OPINS, now assume $\vec{Z}$ is orthonormal. We further
multiply $\vec{Z}$ to both sides of the null-space equation and then
obtain 
\begin{equation}
\vec{Z}\vec{Z}^{T}\vec{A}\vec{Z}\vec{v}=\vec{Z}\vec{Z}^{T}\left(\vec{f}-\vec{A}\vec{x}_{p}\right).\label{eq:OPINSNullSpaceEq2}
\end{equation}
In addition, we rewrite $\vec{x}_{n}=\vec{Z}\vec{v}$ as the orthogonal
projection of a vector $\vec{w}\in\mathbb{R}^{n}$ onto $\mbox{null}\left(\vec{B}\right)$,
i.e., 
\[
\vec{Z}\vec{v}=\vec{Z}\vec{Z}^{T}\vec{w}.
\]
Substituting it into (\ref{eq:OPINSNullSpaceEq2}), we have
\begin{equation}
\vec{Z}\vec{Z}^{T}\vec{A}\vec{Z}\vec{Z}^{T}\vec{w}=\vec{Z}\vec{Z}^{T}\left(\vec{f}-\vec{A}\vec{x}_{p}\right),\label{eq:OPINS_ALT}
\end{equation}
which is equivalent to (\ref{eq:OPINSNullSpaceEq}) in step 4 of Algorithm~\ref{alg:OPINS}. 

The above transformation may seem counterintuitive, as we intentionally
constructed a singular system (\ref{eq:OPINSNullSpaceEq}), which
is larger than the null-space equation (\ref{eq:NullSpaceEq}). However,
(\ref{eq:OPINSNullSpaceEq}) has three key properties: First, it uses
an orthogonal projector to ensure that $\vec{x}_{n}$ is exactly in
$\mbox{null}(\vec{B})$, and hence it overcomes the instability associated
with the projected Krylov methods. Second, since $\vec{\Pi}_{\vec{Z}}=\vec{Z}\vec{Z}^{T}=\vec{I}-\vec{U}\vec{U}^{T}=\vec{\Pi}_{\vec{U}}^{\perp}$,
we can compute the projection by finding $\vec{U}$, which is much
more efficient than finding $\vec{Z}$ when $q\leq m\ll n$. Third,
since (\ref{eq:OPINSNullSpaceEq}) is always singular, and QRCP in
step 1 supports rank-deficient $\vec{B}$, we can apply OPINS to a
saddle-point system regardless of whether $\vec{A}$, $\vec{B}$,
or $\vec{Z}^{T}\vec{A}\vec{Z}$ is singular.

Although the PNS system is singular, it is in fact compatible, due
to the following property.
\begin{lem}
\label{lem:compatability}If the saddle-point system (\ref{eq:KKT})
is compatible in terms of $\vec{f}$, i.e., $\vec{f}\in\mbox{range}(\vec{A})+\mbox{range}(\vec{B}^{T})$,
then the null-space equation (\ref{eq:NullSpaceEq}) and the PNS equation
(\ref{eq:OPINSNullSpaceEq}) are both compatible.
\end{lem}

The proof of this lemma directly follows from the derivations above.
Furthermore, if the null-space equation (\ref{eq:NullSpaceEq}) is
nonsingular, we have the following lemma.
\begin{lem}
\label{lem:gmres}If the null-space equation (\ref{eq:NullSpaceEq})
is nonsingular, then the coefficient matrix of the PNS system (\ref{eq:OPINSNullSpaceEq})
has the same row and column spaces, i.e., $\mbox{range}(\vec{N})=\mbox{range}(\vec{N}^{T})$.
\end{lem}

The proof is straightforward since $\vec{Z}^{T}\vec{A}\vec{Z}$ has
the same row and column spaces. With the compatibility assumption,
Lemma~\ref{lem:gmres} ensures that the standard GMRES can be applied
to system (\ref{eq:OPINSNullSpaceEq}) without breaking down \cite{Reichel2005BGS,Saad86GMRES}.
In the following subsection, we will address the efficient solution
of these compatible systems that may be singular or nonsingular.

\subsubsection{\label{sub:solution-OPNS}Solution of Orthogonally Projected Null-Space
Equation}

Since the PNS system (\ref{eq:NullSpaceEq}) has an infinite number
of solutions, a natural question is which solution of the system suffices
in producing the minimum-norm solution in the sense of (\ref{eq:SingularLangrangeMultiplier}).
In the following, we will address the question first for systems with
a nonsingular null-space equation, followed by more general singular
systems.
\begin{thm}
\label{thm:OPNS1}Given a saddle-point system (\ref{eq:KKT}) where
the null-space equation (\ref{eq:NullSpaceEq}) is nonsingular and
the solution to (\ref{eq:KKT}) is unique in terms of $\vec{x}$,
then the solution of (\ref{eq:OPINSNullSpaceEq}) recovers this unique
$\vec{x}$.\end{thm}
\begin{proof}
Consider the alternative form of the PNS equation in (\ref{eq:OPINS_ALT}).
The system is compatible, so we can always find a solution $\vec{w}\in\mathbb{R}^{n}$
for the equality to hold. Since $\vec{Z}^{T}\vec{Z}=\vec{I}$, left-multiplying
$\vec{Z}^{T}$ on both sides of (\ref{eq:OPINS_ALT}), we obtain the
PNS equation
\[
\vec{Z}^{T}\vec{A}\vec{Z}\vec{Z}^{T}\vec{w}=\vec{Z}^{T}\left(\vec{f}-\vec{A}\vec{x}_{p}\right).
\]
Since system (\ref{eq:NullSpaceEq}) is nonsingular, $\vec{Z}^{T}\vec{w}$
recovers the unique solution for $\vec{v}$, and in turn recovers
the unique $\vec{x}_{n}$ and $\vec{x}$.
\end{proof}

Note that nonsingular (\ref{eq:NullSpaceEq}) includes the cases where
the saddle-point system (\ref{eq:KKT}) is nonsingular. However, it
does not necessarily imply that (\ref{eq:KKT}) is nonsingular, because
$\vec{y}$ may not be unique. An implication of Theorem~\ref{thm:OPNS1}
is that we have the flexibility of solving (\ref{eq:NullSpaceEq})
with any solver even if (\ref{eq:KKT}) is singular. 

For a general singular saddle-point system, the situation is more
complicated. Assume (\ref{eq:KKT}) is compatible in terms of $\vec{f}$,
the following theorem indicates that OPINS finds the minimum-norm
solution of $\vec{x}$.
\begin{thm}
\label{thm:OPNS2}Given a saddle-point system (\ref{eq:KKT}) compatible
in terms of $\vec{f}$, i.e., $\vec{f}\in\mbox{range}(\vec{A})+\mbox{range}(\vec{B}^{T})$;
if $\vec{x}$ is the minimum-norm solution of the PNS equation (\ref{eq:OPINSNullSpaceEq}),
then $\Vert\vec{x}\Vert$ is minimized among all the solutions $\vec{x}=\vec{x}_{n}+\vec{x}_{p}$
that satisfy the constraint $\min_{\vec{x}}\Vert\vec{g}-\vec{B}\vec{x}\Vert$.\end{thm}
\begin{proof}
Note that $\vec{x}=\vec{x}_{p}+\vec{x}_{n}$, where $\vec{x}_{p}\in\mbox{range}(\vec{B}^{T})$
and $\vec{x}_{n}\in\mbox{null}(\vec{B})$, so $\Vert\vec{x}\Vert^{2}=\Vert\vec{x}_{p}\Vert^{2}+\Vert\vec{x}_{n}\Vert^{2}$.
In step 3 of QRCP, $\vec{x}_{p}$ is the minimum-norm solution in
$\mbox{range}(\vec{B}^{T})$ that minimizes $\Vert\vec{g}-\vec{B}\vec{x}\Vert$.
Therefore, we only need to show that $\Vert\vec{x}_{n}\Vert$ is minimized
in $\mbox{null}(\vec{B})$, where $\vec{x}_{n}=\vec{Z}\vec{v}$. This
is satisfied if $\Vert\vec{v}\Vert$ is minimized among the exact
solutions to the null-space equation (\ref{eq:NullSpaceEq}), i.e.,
\[
\vec{Z}^{T}\vec{A}\vec{Z}\vec{v}=\vec{Z}^{T}(\vec{f}-\vec{A}\vec{x}_{p}).
\]
Since $\vec{Z}$ is orthonormal, $\Vert\vec{x}_{n}\Vert=\Vert\vec{v}\Vert$.
In OPINS, if $\vec{w}$ is an exact solution to (\ref{eq:OPINSNullSpaceEq}),
\[
\vec{Z}\vec{Z}^{T}\vec{A}\vec{Z}\vec{Z}^{T}\vec{w}=\vec{Z}\vec{Z}^{T}(\vec{f}-\vec{A}\vec{x}_{p}).
\]
Since $\vec{Z}^{T}\vec{Z}=\vec{I}$, by left-multiplying $\vec{Z}^{T}$
on both sides, we have 
\[
\vec{Z}^{T}\vec{A}\vec{Z}\vec{Z}^{T}\vec{w}=\vec{Z}^{T}(\vec{f}-\vec{A}\vec{x}_{p}),
\]
so $\vec{v}=\vec{Z}^{T}\vec{w}$ is an exact solution of the null-space
equation (\ref{eq:NullSpaceEq}). Note that $\Vert\vec{v}\Vert=\Vert\vec{Z}^{T}\vec{w}\Vert\leq\Vert\vec{Z}^{T}\Vert\Vert\vec{w}\Vert=\Vert\vec{w}\Vert$,
and it is an equality if $\vec{w}\in\mbox{range}(\vec{Z})=\mbox{null}(\vec{B})$,
i.e., $\vec{w}=\vec{Z}\vec{Z}^{T}\vec{w}=\vec{x}_{n}$. Therefore,
the minimum-norm solution of $\vec{w}$ in (\ref{eq:OPINSNullSpaceEq})
minimizes $\Vert\vec{x}_{n}\Vert$. Furthermore, $\vec{g}-\vec{B}\vec{x}=\vec{g}-\vec{B}\vec{x}_{p}$,
so $\Vert\vec{x}\Vert$ is minimized among all solutions under the
constraint $\min_{\vec{x}}\Vert\vec{g}-\vec{B}\vec{x}\Vert$.
\end{proof}

Note that Theorem~\ref{thm:OPNS2} is more general than Theorem~\ref{thm:OPNS1},
as it includes Theorem~\ref{thm:OPNS2} as a special case. An implication
of the two theorems is that if the null-space equation is nonsingular,
we can use any singular solver for OPINS. Examples of such solvers
include MINRES \cite{paige1975solution} and SYMMLQ \cite{paige1975solution}
for symmetric and compatible systems, GMRES \cite{Saad86GMRES}, LSQR
\cite{Paige92LSQR} or LSMR \cite{Fong11LSMR} for nonsymmetric systems.
In addition, we can use any left, right, or symmetric preconditioner
for the solver. However, if the null-space equation is singular,
we must use a solver that can compute the minimum-norm solution for
compatible singular systems, which may exclude GMRES for nonsymmetric
systems if $\mbox{range}(\vec{A})\neq\mbox{range}(\vec{A}^{T})$.
In addition, we can use only left preconditioners that do not alter
the null space of the coefficient matrix. We will further discuss
the preconditioners in Section~\ref{sub:solution-OPNS}.

Finally, regarding the $\vec{y}$ component in (\ref{eq:NullSpaceEq}),
the values of $\vec{y}$ are immaterial for many applications. However,
if desired, we can obtain $\vec{y}$ by solving $\vec{B}^{T}\vec{y}=\left(\vec{f}-\vec{A}\vec{x}\right)$
using QRCP. In general, $\Vert\vec{y}\Vert$ may not be minimized
if $\vec{B}$ is rank deficient, but the norm is typically small.

\subsection{Efficiency of OPINS}

We now analyze the computational cost of OPINS. There are two components
that are relatively more expensive. The first is the QRCP in step
1, which takes $\mathcal{O}(\frac{4}{3}m^{2}n)$ operations when using
Householder transformation. When $m\ll n$, this operation is far
more efficient than finding an orthonormal basis of $\mbox{null}(\vec{B})$.
After obtaining the QR factorization, steps 3, 5, and 6 all take $\mathcal{O}(mn)$
operations. In addition, if $\vec{B}$ is sparse, $\vec{Q}$ and $\vec{R}$
are in general also sparse, leading to even more cost savings. The
other one is the solution of the singular system (\ref{eq:OPINSNullSpaceEq}).
When $\vec{A}$ is large and sparse, it is not advisable to use truncated
SVD or rank-revealing QR factorization for this system. Instead, we
apply a Krylov-subspace method for singular systems. Within each iteration
of these methods, the dominating operation is matrix-vector multiplications,
which cost $\mathcal{O}(N+nq)$, where $N$ denote the total number
of nonzeros in $\vec{A}$. The convergence of these methods depend
on the nonzero eigenvalues \cite{greenbaum1994max}. The following
proposition correlates the eigenvalues of the coefficient matrices
$\hat{\vec{N}}$ and $\vec{N}$ in (\ref{eq:NullSpaceEq}) and (\ref{eq:OPINSNullSpaceEq}),
respectively.
\begin{prop}
Given a saddle-point system (\ref{eq:KKT}), the PNS matrix in (\ref{eq:OPINSNullSpaceEq})
has the same nonzero eigenvalues as the null-space system (\ref{eq:NullSpaceEq})
with an orthonormal $\vec{Z}$.\end{prop}
\begin{proof}
Let $\hat{\vec{N}}=\vec{Z}^{T}\vec{A}\vec{Z}$ and $\vec{N}=\vec{Z}\left(\vec{Z}^{T}\vec{A}\vec{Z}\right)\vec{Z}^{T}=\vec{Z}\hat{\vec{N}}\vec{Z}^{T}$.
If $\hat{\lambda}$ is a nonzero eigenvalue of $\hat{\vec{N}}$, and
$\hat{\vec{x}}$ is a corresponding eigenvalue, then 
\[
\vec{N}(\vec{Z}\hat{\vec{x}})=\left(\vec{Z}\hat{\vec{N}}\vec{Z}^{T}\right)\left(\vec{Z}\hat{\vec{x}}\right)=\vec{Z}\hat{\vec{N}}\hat{\vec{x}}=\hat{\lambda}\vec{Z}\hat{\vec{x}},
\]
so $\hat{\lambda}$ and $\vec{Z}\hat{\vec{x}}$ form a pair of eigenvalue
and eigenvector of $\vec{N}$. Conversely, if $\lambda$ a nonzero
eigenvalue of $\vec{N}$ and $\vec{x}$ is its corresponding eigenvalue,
then 
\[
\hat{\vec{N}}(\vec{Z}^{T}\vec{x})=\left(\vec{Z}^{T}\vec{Z}\right)\left(\vec{Z}^{T}\vec{A}\vec{Z}\right)\vec{Z}^{T}\vec{x}=\vec{Z}^{T}\vec{N}\vec{x}=\lambda\vec{Z}^{T}\vec{x}.
\]
Therefore, $\vec{N}$ and $\hat{\vec{N}}$ have the same nonzero eigenvalues.
\end{proof}

Based on the above proposition, the convergence rate of the Krylov
subspace method on the PNS equation (\ref{eq:OPINSNullSpaceEq}) is
identical to that on the null-space equation (\ref{eq:NullSpaceEq})
with an orthonormal $\vec{Z}$. However, to accelerate the convergence
of these methods, it is desirable to use preconditioners, which we
discuss next.

\section{\label{sec:Preconditioners}Preconditioners for OPINS}

In OPINS, the most time consuming step is typically the iterative
solver for the PNS equation (\ref{eq:OPINSNullSpaceEq}) in step 4.
To speed up its computation, it is critical to use preconditioners.
In this section, we present some principles for constructing effective
preconditioners for PNS equations, based on the recent work on the
projected Krylov methods \cite{gould2014projected}.

Let $\vec{N}=\vec{\Pi}_{\vec{U}}^{\perp}\vec{A}\vec{\Pi}_{\vec{U}}^{\perp}$.
The general idea of preconditioning is to find a matrix $\vec{M}$
that approximates the coefficient matrix $\vec{N}$, or $\vec{M}^{+}$
that approximates the pseudoinverse $\vec{N}^{+}$. The latter form
is more convenient for solving the PNS systems. Algorithm~\ref{alg:OPINS-P}
outlines the pseudocode of the preconditioned OPINS with a left preconditioner.
The preconditioning routine takes an operator $\vec{M}$ to evaluate
$\vec{M}^{+}\vec{b}$ for any $\vec{b}\in\mathbb{R}^{n}$. Note that
for symmetric systems, most preconditioned Krylov-subspace methods
would apply the preconditioners symmetrically. Specifically, suppose
$\vec{M}^{+}$ has a symmetric factorization $\vec{M}^{+}=\vec{L}\vec{L}^{T}$,
then these methods solve the equation 
\[
\vec{L}^{T}\vec{\Pi}_{\vec{U}}^{\perp}\vec{A}\vec{\Pi}_{\vec{U}}^{\perp}\vec{L}\tilde{\vec{w}}=\vec{L}^{T}\vec{\Pi}_{\vec{U}}^{\perp}\left(\vec{f}-\vec{A}\vec{x}_{p}\right)
\]
in the preconditioned method, and then computes $\vec{w}=\vec{L}\tilde{\vec{w}}$.
Typically, the algorithm is constructed such that the explicit factorization
$\vec{M}^{+}=\vec{L}\vec{L}^{T}$ is not needed. We omit the details
of such preconditioned Krylov-subspace methods; interested readers
may refer to \cite{BBC94Templates,Saad03IMS}.

\begin{algorithm}
\protect\caption{\label{alg:OPINS-P}Preconditioned OPINS}

\textbf{input}: $\vec{A}$, $\vec{B}$, $\vec{f}$, $\vec{g}$, $\vec{G}$,
tolerances for rank estimation and iterative solver

\textbf{output}: $\vec{x}$, $\vec{y}$ (optional)

\begin{algorithmic}[1]

\STATE do first three steps of Algorithm ~\ref{alg:OPINS}

\STATE solve $\vec{M}^{+}\vec{\Pi}_{\vec{U}}^{\perp}\vec{A}\vec{\Pi}_{\vec{U}}^{\perp}\vec{w}=\vec{M}^{+}\vec{\Pi}_{\vec{U}}^{\perp}\left(\vec{f}-\vec{A}\vec{x}_{p}\right)$
using a preconditioned Krylov-subspace method

\STATE do the last two steps of Algorithm ~\ref{alg:OPINS}.

\end{algorithmic}
\end{algorithm}

In this section, we will focus on the construction of $\vec{M}$ or
$\vec{M}^{+}$. A straightforward choice of $\vec{M}$ is the approximation
of $\vec{A}$. Possible candidates include SSOR-type preconditioners,
incomplete factorization, and multigrid methods. We propose to approximate
$\vec{N}^{+}$ with $\vec{M}^{+}=\vec{P}_{G}=\vec{Z}\left(\vec{Z}^{T}\vec{G}\vec{Z}\right)^{-1}\vec{Z}^{T}$,
where $\vec{G}$ is an approximation of $\vec{A}$ and $\vec{Z}^{T}\vec{G}\vec{Z}$
is nonsingular. We refer to this preconditioner as the \emph{projected
preconditioner}. As we shall show in this section, for nonsingular
systems it is equivalent to the constraint preconditioner in the projected
Krylov methods \cite{gould2014projected,keller2000constraint}. In
the following, we will first discuss the details of implementing $\vec{P}_{G}$
as an operator for nonsymmetric systems, and then specialize it for
symmetric systems.

\subsection{Preconditioners for Nonsymmetric Systems}

If the saddle-point system is nonsymmetric, the projected preconditioner
applies as long as $\vec{G}$ is nonsingular. For example, we can
take $\vec{G}$ as the SOR-style preconditioner or incomplete LU factorization
of $\vec{A}$. Then, $\vec{P}_{G}=\vec{Z}\left(\vec{Z}^{T}\vec{G}\vec{Z}\right)^{-1}\vec{Z}^{T}$,
and we use $\vec{P}_{G}$ as a left preconditioner. The following
proposition states that OPINS with the projected preconditioner is
equivalent to applying $\vec{Z}^{T}\vec{G}\vec{Z}$ as a left preconditioner
to a null-space equation, and hence is equivalent to the constraint-preconditioned
null-space methods \cite{gould2014projected,keller2000constraint}.
\begin{prop}
\label{prp:precond-1}Assume $\vec{Z}^{T}\vec{G}\vec{Z}$ is nonsingular.
OPINS with the projected preconditioner as a left preconditioner is
equivalent to left preconditioning the null-space equation with $\vec{Z}^{T}\vec{G}\vec{Z}$.\end{prop}
\begin{proof}
Since $\vec{P}_{G}$ is applied as a left preconditioner, the preconditioned
OPINS solves 
\[
\underbrace{\vec{Z}\left(\vec{Z}^{T}\vec{G}\vec{Z}\right)^{-1}\vec{Z}^{T}}_{\vec{P}_{G}}\underbrace{\vec{Z}\vec{Z}^{T}\vec{A}\vec{Z}\vec{Z}^{T}}_{\vec{N}}\vec{w}=\underbrace{\vec{Z}\left(\vec{Z}^{T}\vec{G}\vec{Z}\right)^{-1}\vec{Z}^{T}}_{\vec{P}_{G}}\vec{Z}\vec{Z}^{T}\left(\vec{f}-\vec{A}\vec{x}_{p}\right).
\]
Since $\vec{Z}^{T}\vec{Z}=\vec{I}$ and $\vec{Z}$ is orthonormal,
it is equivalent to solving 
\[
\left(\vec{Z}^{T}\vec{G}\vec{Z}\right)^{-1}\vec{Z}^{T}\vec{A}\vec{Z}\vec{v}=\left(\vec{Z}^{T}\vec{G}\vec{Z}\right)^{-1}\vec{Z}^{T}\left(\vec{f}-\vec{A}\vec{x}_{p}\right),
\]
i.e., applying $\vec{Z}^{T}\vec{G}\vec{Z}$ as a left preconditioner
to (\ref{eq:NullSpaceEq}), where $\vec{v}=\vec{Z}^{T}\vec{w}$.
\end{proof}

As shown in \cite{keller2000constraint}, the eigenvalues of the constraint-preconditioned
null-space equation are well clustered if $\vec{G}$ is a good approximation
of $\vec{A}$. Due to the above equivalence, the projected preconditioner
is a good choice for OPINS. Note that when applying $\vec{P}_{G}$
as a left preconditioner, it does not alter the null-space of $\vec{N}$,
as claimed in the following proposition.
\begin{prop}
\label{prp:left-precond}Assume $\vec{Z}^{T}\vec{G}\vec{Z}$ is nonsingular.
OPINS with the projected preconditioner as the left preconditioner
does not alter the null-space of $\vec{N}$, i.e., $\mbox{null}(\vec{P}_{G}\vec{N})=\mbox{null}(\vec{N})$.\end{prop}
\begin{proof}
For any vector $\vec{v}\in\mbox{null}(\vec{N})$, $\vec{v}\in\mbox{null}(\vec{P}_{G}\vec{N})$.
If $\vec{v}\in\mbox{null}(\vec{P}_{G}\vec{N})$, we have
\[
\vec{Z}\left(\vec{Z}^{T}\vec{G}\vec{Z}\right)^{-1}\vec{Z}^{T}\vec{A}\vec{Z}\vec{Z}^{T}\vec{v}=0.
\]
Since $\vec{Z}$ has full column rank and $\vec{Z}^{T}\vec{G}\vec{Z}$
is nonsingular, $\vec{Z}^{T}\vec{A}\vec{Z}\vec{Z}^{T}\vec{v}=0$.
It follows that $\vec{v}\in\mbox{null}(\vec{N})$.
\end{proof}

The above proposition indicates that we can apply the projected preconditioner
as a left preconditioner to singular saddle-point that is compatible
in $\vec{f}$, while ensuring $\vec{x}$ has minimum norm.

The remaining task is to find a way to provide $\vec{P}_{G}$ as an
operator for efficient computation of $\vec{s}=\vec{P}_{G}\vec{b}$
for any $\vec{b}\in\mathbb{R}^{n}$. Note that $\vec{s}$ is the solution
to the following modified but simpler saddle-point system 
\begin{equation}
\begin{bmatrix}\vec{G} & \vec{U}\\
\vec{U}^{T} & \vec{0}
\end{bmatrix}\begin{bmatrix}\vec{s}\\
\vec{t}
\end{bmatrix}=\begin{bmatrix}\vec{b}\\
\vec{0}
\end{bmatrix},\label{eq:constraint_preconditioner}
\end{equation}
where $\vec{U}$ is composed of an orthonormal basis of $\mbox{range}(\vec{B}^{T})$.
Because from the null-space method, we have
\begin{equation}
\vec{s}=\vec{Z}\left(\vec{Z}^{T}\vec{G}\vec{Z}\right)^{-1}\vec{Z}^{T}\vec{b}=\vec{P}_{G}\vec{b}.\label{eq:constraint_pj2}
\end{equation}
This implies that $\vec{P}_{G}$ can be given as an operator through
the solution of (\ref{eq:constraint_preconditioner}). (\ref{eq:constraint_prec})
is similar to the procedure for evaluating the constraint preconditioner
in the projected Krylov method \cite{gould2014projected}, for which
the off-diagonal entries are $\vec{B}^{T}$ and $\vec{B}$ instead
of $\vec{U}$ and $\vec{U}^{T}$. By replacing $\vec{B}^{T}$ by $\vec{U}$,
(\ref{eq:constraint_prec}) can be solved more efficiently, and it
is also applicable if $\vec{B}$ is rank deficient.

If $\vec{G}$ is a simple matrix or operator, such as the preconditioner
based on SOR or incomplete factorization, we can solve (\ref{eq:constraint_prec})
efficiently by using the range-space method, given by (\ref{eq:RangeSpace1})
and (\ref{eq:RangeSpace2}), especially when $m\ll n$. Algorithm~\ref{alg:PK}
outlines the procedure for computing $\vec{P}_{G}\vec{b}$ using the
range-space method, where $\vec{U}$ is composed of an orthonormal
basis of $\vec{B}^{T}$.

\begin{algorithm}
\protect\caption{\label{alg:PK}Operator $\vec{P}_{G}$ for Projected Preconditioner}

\textbf{input}: $\vec{G}$, $\vec{U}$, $\vec{b}$

\textbf{output}: $\vec{s}=\vec{P}_{G}\vec{b}$

\begin{algorithmic}[1]

\STATE solve $\vec{G}\vec{r}=\vec{b}$

\STATE solve $(\vec{U}^{T}\vec{G}^{-1}\vec{U})\vec{t}=\vec{U}^{T}\vec{r}$

\STATE solve $\vec{G}\vec{s}=\vec{b}-\vec{U}\vec{t}$

\end{algorithmic}
\end{algorithm}

\subsection{Preconditioners for Symmetric Systems}

For symmetric systems, most preconditioned Krylov subspace methods
require the preconditioner to be symmetric. Hence, we require that
$\vec{G}$ is symmetric and $\vec{Z}^{T}\vec{G}\vec{Z}$ is SPD. Therefore,
$\vec{P}_{G}$ is symmetric and positive semi-definite. If the preconditioner
is applied as a left preconditioner, the result for nonsymmetric systems
also apply to symmetric systems. Assuming the system is symmetric,
the following proposition states that OPINS with the projected preconditioner
is equivalent to applying $\vec{Z}^{T}\vec{G}\vec{Z}$ as a preconditioner
in solving the corresponding null-space equation, regardless whether
it is applied as a left preconditioner or a symmetric preconditioner.
\begin{prop}
\label{prp:precond}Assume $\vec{Z}^{T}\vec{G}\vec{Z}$ is SPD and
the saddle-point system is symmetric. OPINS with the projected preconditioner
is equivalent to applying $\vec{Z}^{T}\vec{G}\vec{Z}$ as the preconditioner
for solving the null-space equation in the null-space method.\end{prop}
\begin{proof}
If $\vec{P}_{G}$ is applied as a left preconditioner, the proof is
the same as for nonsymmetric systems. Now consider $\vec{Z}^{T}\vec{G}\vec{Z}$
is SPD, then so is $\left(\vec{Z}^{T}\vec{G}\vec{Z}\right)^{-1}$,
which has a Cholesky factorization 
\[
\left(\vec{Z}^{T}\vec{G}\vec{Z}\right)^{-1}=\vec{L}_{G}\vec{L}_{G}^{T}.
\]
Then $\vec{P}_{G}=\vec{Z}\left(\vec{Z}^{T}\vec{G}\vec{Z}\right)^{-1}\vec{Z}^{T}$
has a symmetric factorization
\[
\vec{P}_{G}=\vec{L}\vec{L}^{T}=\vec{Z}\vec{L}_{G}(\vec{Z}\vec{L}_{G})^{T}=\vec{Z}\vec{L}_{G}\vec{L}_{G}^{T}\vec{Z}^{T},
\]
where $\vec{L}=\vec{Z}\vec{L}_{G}$. If $\vec{P}_{G}$ is applied
symmetrically, the preconditioned OPINS would solve the equation 
\[
\underbrace{\vec{L}_{G}^{T}\vec{Z}^{T}}_{\vec{L}^{T}}\underbrace{\vec{Z}\vec{Z}^{T}\vec{A}\vec{Z}\vec{Z}^{T}}_{\vec{N}}\underbrace{\vec{Z}\vec{L}_{G}}_{\vec{L}}\tilde{\vec{w}}=\underbrace{\vec{L}_{G}^{T}\vec{Z}^{T}}_{\vec{L}^{T}}\vec{Z}\vec{Z}^{T}\left(\vec{f}-\vec{A}\vec{x}_{p}\right),
\]
where $\vec{Z}^{T}\vec{Z}=\vec{I}$ and $\vec{w}=\vec{L}_{G}\tilde{\vec{w}}$.
Therefore, it is equivalent to solving 
\[
\vec{L}_{G}^{T}\underbrace{\vec{Z}^{T}\vec{A}\vec{Z}}_{\hat{\vec{N}}}\vec{L}_{G}\tilde{\vec{w}}=\vec{L}_{G}^{T}\vec{Z}^{T}\left(\vec{f}-\vec{A}\vec{x}_{p}\right),
\]
which is equivalent to applying $\vec{Z}^{T}\vec{G}\vec{Z}$ as a
symmetric preconditioner to (\ref{eq:NullSpaceEq}).
\end{proof}

For preconditioned Krylov subspace methods for symmetric systems,
$\vec{P}_{G}$ can be computed as an operator using Algorithm~\ref{alg:PK}.
Note that if the system is singular, applying the preconditioner symmetrically
may alter the null space of the coefficient matrix, unless it is equivalent
to applying a left preconditioner. Therefore, additional care must
be taken to find the minimum-norm solution when applying a preconditioner
symmetrically.

\section{\label{sec:Results}Numerical Results}

In this section, we evaluate OPINS with various test problems, including
both singular and nonsingular systems. The experiments are mainly
focused on symmetric systems. Some results of nonsymmetric systems
are also included in Section~\ref{sub:opins_precond}. We start by
evaluating the performance of OPINS with and without preconditioners,
to demonstrate the importance of preconditioners and the effectiveness
of the projected preconditioner. We then compare OPINS against some
present state-of-the-art methods for symmetric nonsingular and singular
systems. We use a few test problems, as summarized in Table~\ref{tab:test-matrices}
arising from constrained minimization, finite element analysis, climate
modeling, and random matrices. Among these problems, the first six
are sparse, and Figure~\ref{fig:Sparsity-patterns} shows the sparsity
patterns for some of their KKT matrices. If the right-hand sides were
unavailable, we generate them by multiplying the matrix with a random
vector. For all problems, we set the convergence tolerance to $10^{-10}$
for the residual in iterative methods, and set the tolerance for QRCP
to $10^{-12}$. 

\begin{table}
\protect\caption{\label{tab:test-matrices}Summary of test problems.}

\centering{}%
\begin{tabular}{|c|c|c|c|l|}
\hline 
problem & $\mbox{len}(\vec{x})$ & $\mbox{len}(\vec{y})$ & rank$(\vec{K})$ & source\tabularnewline
\hline 
\hline 
\textsf{3d-var} & 3240 & 3 & 3243 & Analysis for climate modeling \cite{Hthesis}\tabularnewline
\hline 
\textsf{sherman5} & 3312 & 20 & 3332 & nonsymmetric problem from \cite{Boisvert:1997:MMW:265834.265854}\tabularnewline
\hline 
\textsf{mosarqp1} & 2500 & 700 & 3200 & quadratic programming \cite{maros1999repository}\tabularnewline
\hline 
\textsf{fracture} & 780 & 92 & 786 & 2-D elasticity with fracture \cite{bathe1975finite}\tabularnewline
\hline 
\textsf{can\_61} & 61 & 20 & 81 & symmetric problem from \cite{Boisvert:1997:MMW:265834.265854}\tabularnewline
\hline 
\textsf{random} & 100 & 20 & 120 & random nonsingular matrix\tabularnewline
\hline 
\textsf{random-s} & 100 & 20 & 90 & random singular matrix\tabularnewline
\hline 
\end{tabular}
\end{table}

\begin{figure}
\begin{centering}
\begin{minipage}[t]{0.33\columnwidth}%
\begin{center}
\includegraphics[width=0.9\columnwidth]{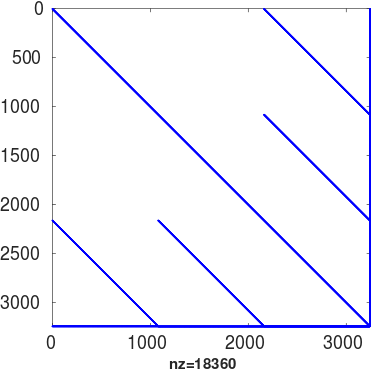}\\
{\small{}(a) }\textsf{3d-var}
\par\end{center}%
\end{minipage}%
\begin{minipage}[t]{0.33\columnwidth}%
\begin{center}
\includegraphics[width=0.9\columnwidth]{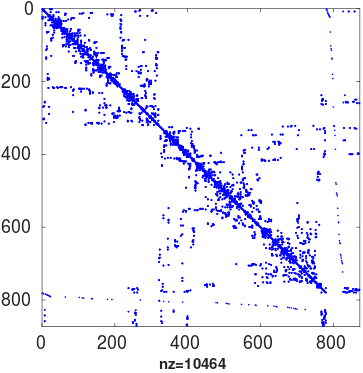}\\
{\small{}(b) }\textsf{fracture}
\par\end{center}%
\end{minipage}%
\begin{minipage}[t]{0.33\columnwidth}%
\begin{center}
\includegraphics[width=0.9\columnwidth]{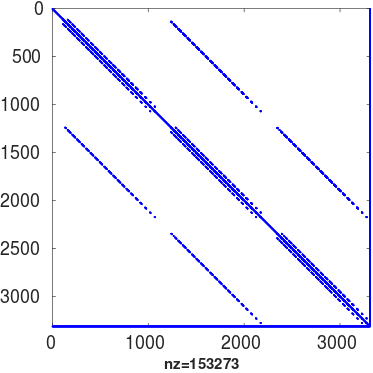}\\
{\small{}(c) }\textsf{sherman5}
\par\end{center}%
\end{minipage}
\par\end{centering}

\protect\caption{\label{fig:Sparsity-patterns}Sparsity patterns of KKT matrices in
\textsf{\small{}3d-var}, \textsf{fracture} and \textsf{sherman5}.}
\end{figure}

In terms of error measures, different methods use different convergence
criteria internally. For a direct comparison, we present the convergence
results in two difference measures. The first is the residual of $\vec{x}$
within the null space of $\vec{B}$, i.e., 
\[
\vec{r}:=\vec{\Pi}_{\vec{Z}}\left(\vec{f}-\vec{A}\vec{x}_{p}\right)-\vec{\Pi}_{\vec{Z}}\vec{A}\vec{x}_{n},
\]
where $\vec{x}_{p}$ is the particular solution in $\mbox{range}(\vec{B}^{T})$
and $\vec{x}_{n}$ is the corresponding solution in $\mbox{null}(\vec{B})$.
In the context of constraint minimization, this residual is an indicator
of how well $\nabla_{\vec{d}}\phi=\vec{0}$ is satisfied for the objective
function $\phi$ for all directions within the constraint hyperplane.
To make the metric scale independent, we measure the residual relative
to the right-hand side of (\ref{eq:NullSpaceEq}), i.e., 
\[
\mbox{relative residual in \ensuremath{\vec{x}}}:=\left\Vert \vec{r}\right\Vert /\left\Vert \vec{\Pi}_{\vec{Z}}\left(\vec{f}-\vec{A}\vec{x}_{p}\right)\right\Vert .
\]
When comparing OPINS with methods that solve for $\vec{x}$ and $\vec{y}$
simultaneously, such as preconditioned Krylov methods, we calculate
the residual in $\vec{x}$ by computing $\vec{U}$ using QR factorization.
For a more complete comparison, in addition to the above error metric,
we also compute the residual for the whole system (\ref{eq:KKT})
relative to the right-hand side in terms of both $\vec{x}$ and $\vec{y}$.

\subsection{\label{sub:opins_precond}Effectiveness of Preconditioners}

For Krylov subspace methods, the preconditioners have significant
impact on the convergence rate. First let us consider symmetric systems.
The core solver in OPINS is based on MINRES, so we assess the effectiveness
of OPINS with a straightforward preconditioner as well as the projected
preconditioner as described in Section~\ref{sec:Preconditioners}.
For simplicity, we choose $\vec{G}$ as the Jacobi preconditioner
for both preconditioners, and denote them as OPINS-J and OPINS-P,
respectively. 

We solve the test case \textsf{3d-var} with unpreconditioned OPINS,
OPINS-J and OPINS-P. Figure~\ref{fig:opins_cmp_mosarqp}(a) shows
the convergence results measured in terms of the $\vec{x}$ residual.
For \textsf{3d-var}, the block $\vec{A}$ is nearly diagonal and is
strongly diagonal dominant, so both preconditioners worked well and
performed significantly better than unpreconditioned OPINS. Between
the preconditioners, the projected preconditioner performed better
than the Jacobi preconditioner for MINRES alone, because the projected
preconditioner provides a better approximation to the whole matrix.
This indicates that the projected preconditioner is effective for
accelerating OPINS, if $\vec{G}$ is a good approximation of $\vec{A}$. 

To demonstrate the applicability to nonsymmetric systems, we solve
the problem \textsf{sherman5} from the matrix market database \cite{Boisvert:1997:MMW:265834.265854}.
$\vec{A}$ comes from an oil reservoir simulation and $\vec{B}$ is
a random matrix. Since the system is nonsymmetric, we choose $\vec{G}$
as the ILU factorization of $\vec{A}$ for both preconditioners, and
denote the two strategies as OPINS-ILU and OPINS-P, respectively.
The inner solver is GMRES with the number of restarts set to 50. Figure~\ref{fig:opins_cmp_mosarqp}(b)
shows the convergence results for the three approaches. The results
show that OPINS indeed works for nonsymmetric systems. Similar to
the symmetric case, OPINS-P is faster than OPINS-ILU while both are
much faster than unpreconditioned OPINS.\textsf{ }

\begin{figure}
\begin{centering}
\begin{minipage}[t]{0.49\columnwidth}%
\begin{center}
\includegraphics[width=1\columnwidth]{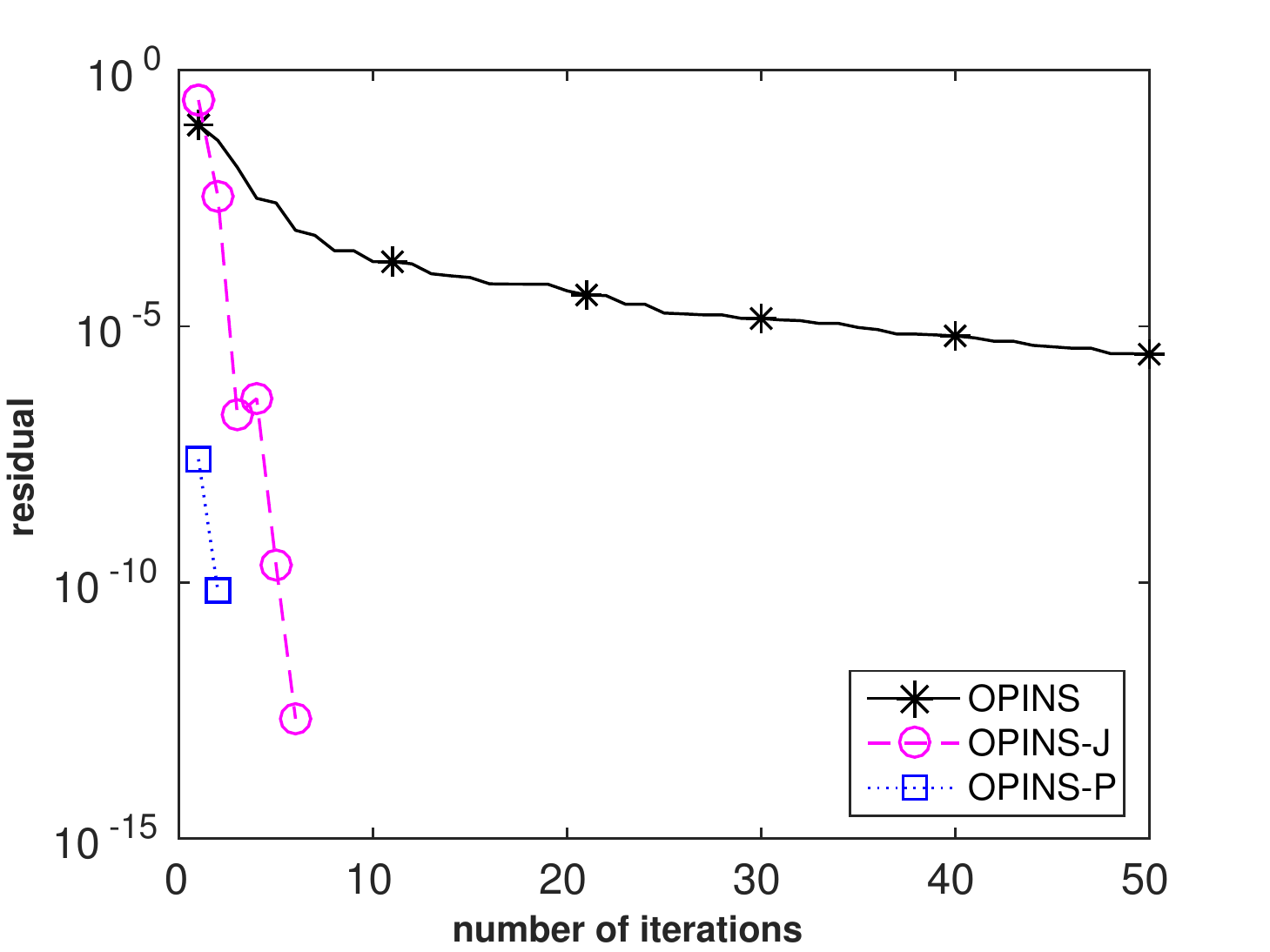}\\
{\small{}(a) }\textsf{3d-var}
\par\end{center}%
\end{minipage}%
\begin{minipage}[t]{0.49\columnwidth}%
\begin{center}
\includegraphics[width=1\columnwidth]{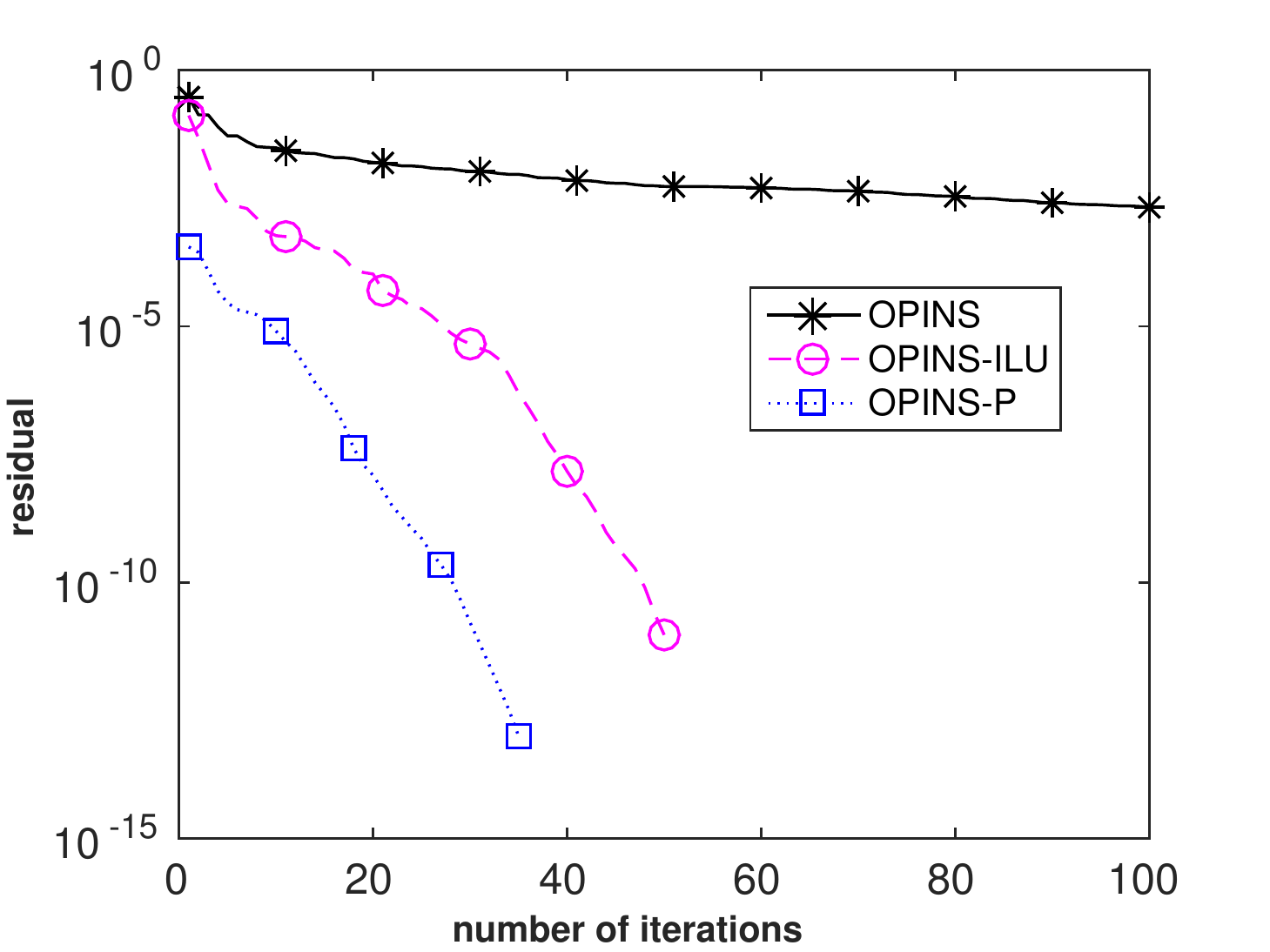}\\
{\small{}(b)}\textsf{ sherman5}
\par\end{center}%
\end{minipage}
\par\end{centering}

\protect\caption{\label{fig:opins_cmp_mosarqp} Convergence history of relative residual
in $\vec{x}$ of OPINS versus preconditioned OPINS for\textsf{ 3d-var
}and\textsf{ sherman5.}}
\end{figure}

\subsection{\label{sub:Nonsingular}Assessment for Symmetric Nonsingular Systems}

We now perform a more in-depth assessment of OPINS for symmetric nonsingular
systems. Since unpreconditioned OPINS is not effective for nonsingular
systems, we only consider OPINS-J and OPINS-P. We compare them against
other Krylov subspace methods, including PMINRES and PCG with constraint
preconditioners, which are equivalent to implicit null-space methods
and are effective choices for the saddle-point systems that we are
considering. As a reference, we also consider the preconditioned GMRES
with the constraint preconditioner applied to the original system
(\ref{eq:KKT}). The number of restarts is set to 50. We calculated
the error for GMRES only after the solver has converged, so only one
data point is available for it.

We solve three test problems, \textsf{mosarqp1,} \textsf{3d-var and
random.} Figure~\ref{fig:overall_cmp}(a) and (b) show the convergence
of the residual in terms of $\vec{x}$ for \textsf{3d-var }and \textsf{random}.
In \textsf{3d-var}, all methods perform well with OPINS-P being similar
to PMINRES. In terms of solution accuracy, all strategies give accurate
$[\vec{x},\vec{y}]$ as shown in Table~\ref{tab:whole_res_ns}. 

\begin{figure}
\begin{centering}
\begin{minipage}[t]{0.5\columnwidth}%
\begin{center}
\includegraphics[width=1\columnwidth]{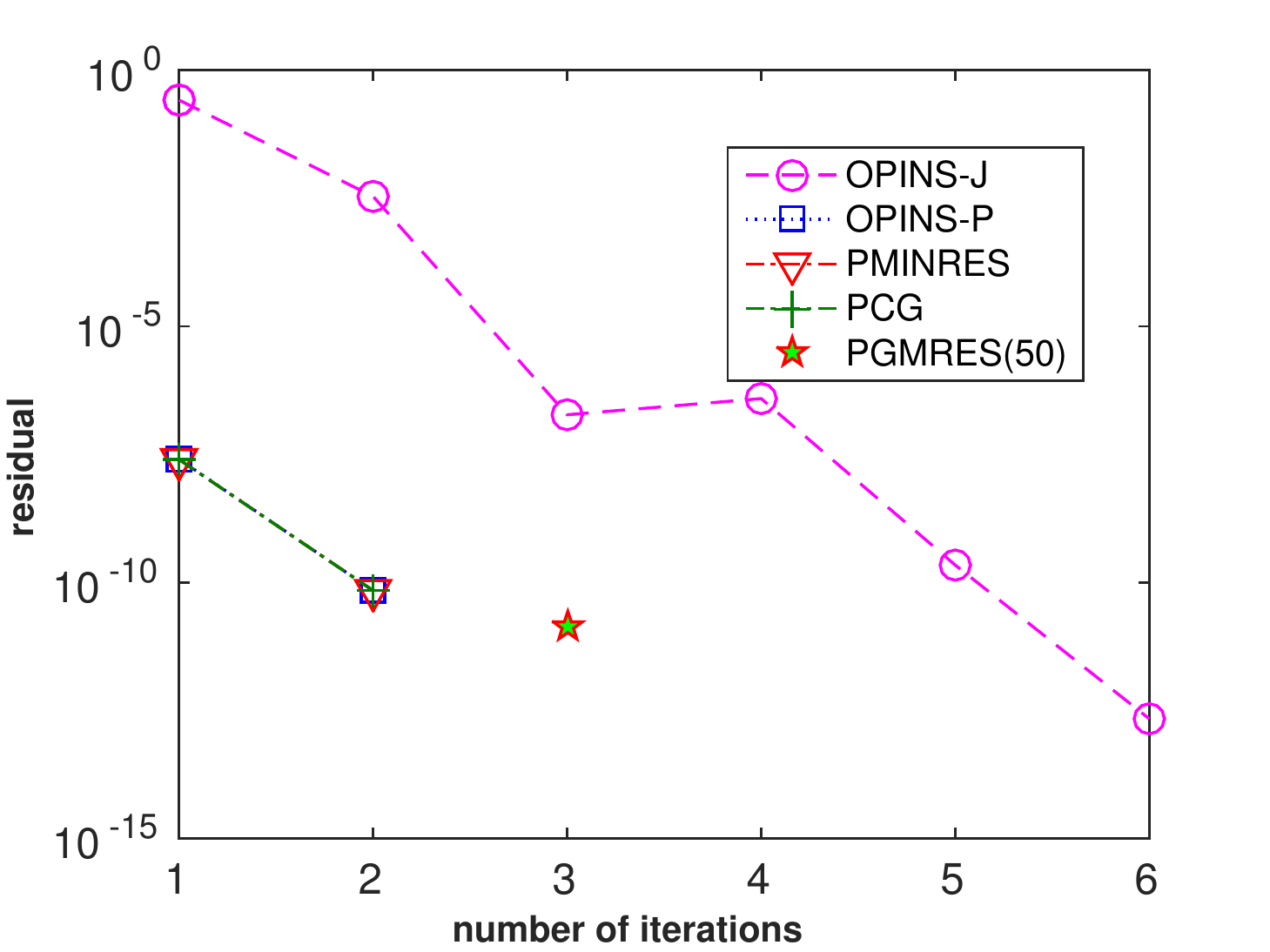}\\
{\small{}(a) }\textsf{3d-var}{\small{}.}
\par\end{center}%
\end{minipage}%
\begin{minipage}[t]{0.5\columnwidth}%
\begin{center}
\includegraphics[width=1\columnwidth]{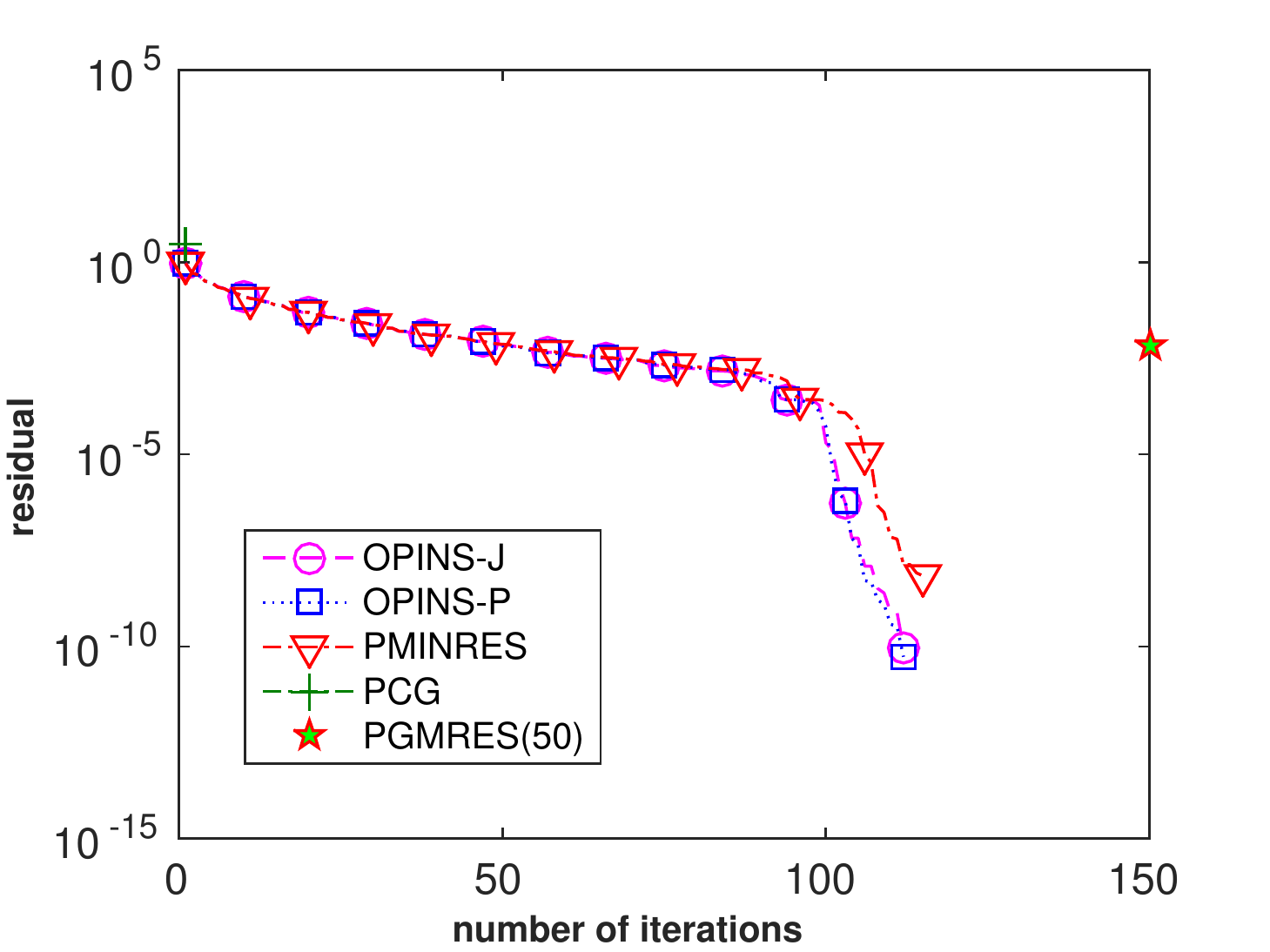}\\
{\small{}(b) }\textsf{\small{}random}{\small{}.}
\par\end{center}%
\end{minipage}
\par\end{centering}

\protect\caption{\label{fig:overall_cmp} Convergence history of relative residuals
in $\vec{x}$ for OPINS and projected Krylov methods for \textsf{3d-var}
and \textsf{random}.}
\end{figure}

\begin{table}
\protect\caption{\label{tab:whole_res_ns}Relative residual in $[\vec{x},\vec{y}]$
upon convergence for OPINS and convergence or stagnation of projected
Krylov methods.}

\centering{}%
\begin{tabular}{|c|c|c|c|c|c|}
\hline 
Problem & OPINS-J & OPINS-P & PMINRES & PCG & PGMRES(50)\tabularnewline
\hline 
\hline 
\textsf{3d-var} & 7.8e-13 & 2.5e-10 & 2.5e-10 & 2.5e-10 & 1.3e-11\tabularnewline
\hline 
\textsf{mosarqp1} & 2.1e-11 & 3.9e-11 & 1.3e-10 & 1.3e-10 & 5.4e-11\tabularnewline
\hline 
\textsf{random} & 1.2e-12 & 1.2e-12 & 3.5e-4 & break & \footnotemark[1]1.5e-4\tabularnewline
\hline 
\end{tabular}
\end{table}

In the \textsf{random} test, the $\vec{A}$ block is a $100\times100$
symmetric, nonsingular and indefinite matrix. $\vec{B}$ is a randomly
generated $20\times100$ dense matrix with full rank. Figure~\ref{fig:overall_cmp}(b)
shows the $\vec{x}$ errors of different strategies. In this example,
$\vec{A}$ is no longer positive definite, and PCG breaks at the initial
iterations. Compared to PGMRES, OPINS and PMINRES converged much faster
both in terms of $\vec{x}$ and in terms of $[\vec{x},\vec{y}]$.
One reason is that OPINS and PMINRES are symmetric methods that can
take advantage of a complete Krylov subspace basis. On the other hand,
the $\vec{x}$ residual of PMINRES will stagnate around $10^{-8}$.
This is due to the instability of the algorithm as mentioned in \cite{gould2014projected};
see Subsection~\ref{sub:Stability} for a more detailed analysis
and comparison. Due to its stability and faster convergence, OPINS
delivered the most accurate solution among all the methods both in
terms of $\vec{x}$ and in terms of $[\vec{x},\vec{y}]$.

\footnotetext[1]{This is the error for PGMRES after 150 iterations.}

\subsubsection{\label{sub:Stability}Stability of OPINS versus PMINRES}

In this section, we study the stability of OPINS and PMINRES. The
difference between PCG and PMINRES with a constraint preconditioner
mainly lies in the underlying solver. Since MINRES is more robust
than CG, we restrict our attention to PMINRES. In our previous discussion,
PMINRES and OPINS with the projected preconditioner are both equivalent
to a preconditioned null-space method. However PMINRES may stagnate
for some problems, as shown in Figure~\ref{fig:overall_cmp} (b). 

To further illustrate this, we consider the \textsf{random} test used
in Subsection~\ref{sub:Nonsingular}. OPINS with the two preconditioners
are applied here. PMINRES uses the constraint preconditioner with
$\vec{G}$ chosen as the diagonal part of $\vec{A}$. To examine the
stability of PMINRES, we consider PMINRES+IR(0) and PMINRES+IR(1),
which denote no iterative refinement and one step of iterative refinement,
respectively. The constraint preconditioner is solved by factorization.
Figure~\ref{fig:comparison2}(a) shows the convergence of various
strategies. It can be observed that OPINS-J, OPINS-P and PMINRES+IR(1)
have similar convergence behaviors. On the other hand, PMINRES failed
to converge to the desired tolerance if no iterative refinement is
applied. Another example is the fracture problem. This system is singular,
but the projected MINRES is also applicable. We set $\vec{G}$ to
be the identity matrix in the constraint preconditioner. No preconditioner
is used in OPINS. From Figure~\ref{fig:comparison2}(b), we can see
that OPINS is similar to the more stable PMINRES+IR(1), while the
residual of PMINRES+IR(0) oscillates.

\begin{figure}
\begin{centering}
\begin{minipage}[t]{0.5\columnwidth}%
\begin{center}
\includegraphics[width=1\columnwidth]{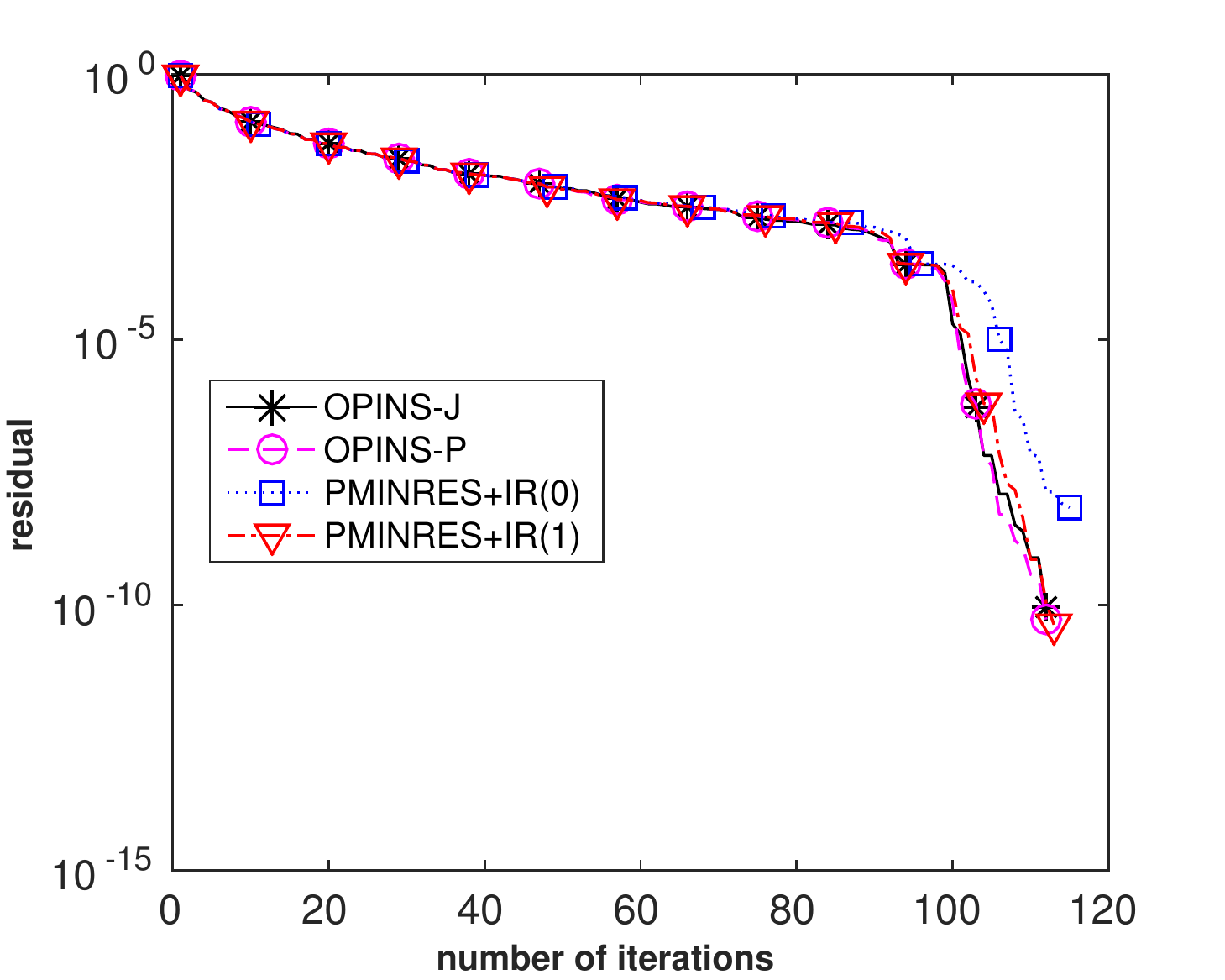}\\
{\small{}(a) Relative residuals in $\vec{x}$ for }\textsf{\small{}random}{\small{}.}
\par\end{center}%
\end{minipage}%
\begin{minipage}[t]{0.5\columnwidth}%
\begin{center}
\includegraphics[width=1\columnwidth]{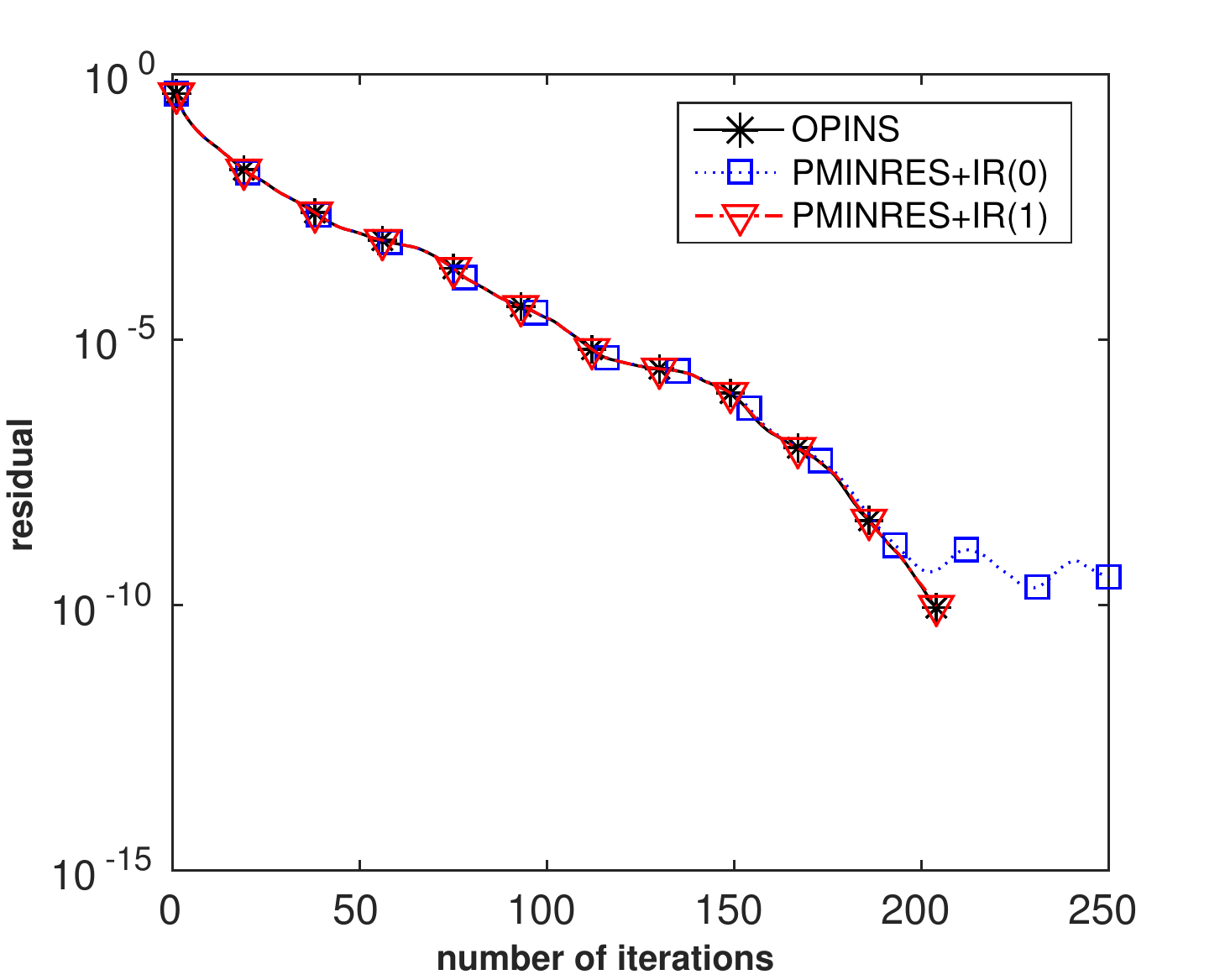}\\
{\small{}(b) Relative residuals in $\vec{x}$ for }\textsf{\small{}fracture}{\small{}.}
\par\end{center}%
\end{minipage}
\par\end{centering}

\protect\caption{\label{fig:comparison2} Comparison of OPINS and projected MINRES
with/without iterative refinement for \textsf{random} and \textsf{fracture}.}
\end{figure}

In addition, OPINS is more stable compared to PMINRES with iterative
refinement, because the latter may still stagnate. We consider the
example \textsf{can\_61} from the matrix market database \cite{Boisvert:1997:MMW:265834.265854}.
The matrix is set as the $\vec{A}$ part of the saddle-point system.
$\vec{B}$ is a randomly generated $20\times61$ dense matrix with
full rank. In this example, we include an additional strategy, PMINRES+IR(2)
which uses two steps of iterative refinement per iteration. Figure~\ref{fig:comparison3}(a)
displays the residuals of the five strategies. Both OPINS-J and OPINS-P
converged to the specified accuracy. However, PMINRES+IR(0) and PMINRES+IR(1)
both stagnated early. Even additional iterative refinement could not
increase the accuracy further as PMINRES+IR(2) still stagnated. 

One reason for the stagnation of PMINRES is that $\vec{x}_{n}$ in
the constraint preconditioner may deviate from $\mbox{null}\left(\vec{B}\right)$
during the iteration. Eventually the large residual in $\vec{B}\vec{x}_{n}$
may cause MINRES to stagnate \cite{gould2014projected}. This is evident
in Figure~\ref{fig:comparison3}(b), where we plot $\Vert\vec{B}\vec{x}_{n}\Vert$
and $\left\Vert \vec{B}\vec{\Pi}_{\vec{U}}^{\perp}\vec{w}\right\Vert $
for PMINRES and OPINS. For PMINRES with and without iterative refinement,
$\Vert\vec{B}\vec{x}_{n}\Vert$ grew, and the growth of non null-space
component was significant even with two steps of iterative refinement.
In contrast, this norm stayed at close to machine precision for OPINS
throughout the computation, because OPINS explicitly projects the
solution onto $\mbox{null}\left(\vec{B}\right)$. For this reason,
OPINS is more stable than PMINRES.

\begin{figure}
\begin{centering}
\begin{minipage}[t]{0.5\columnwidth}%
\begin{center}
\includegraphics[width=1\columnwidth]{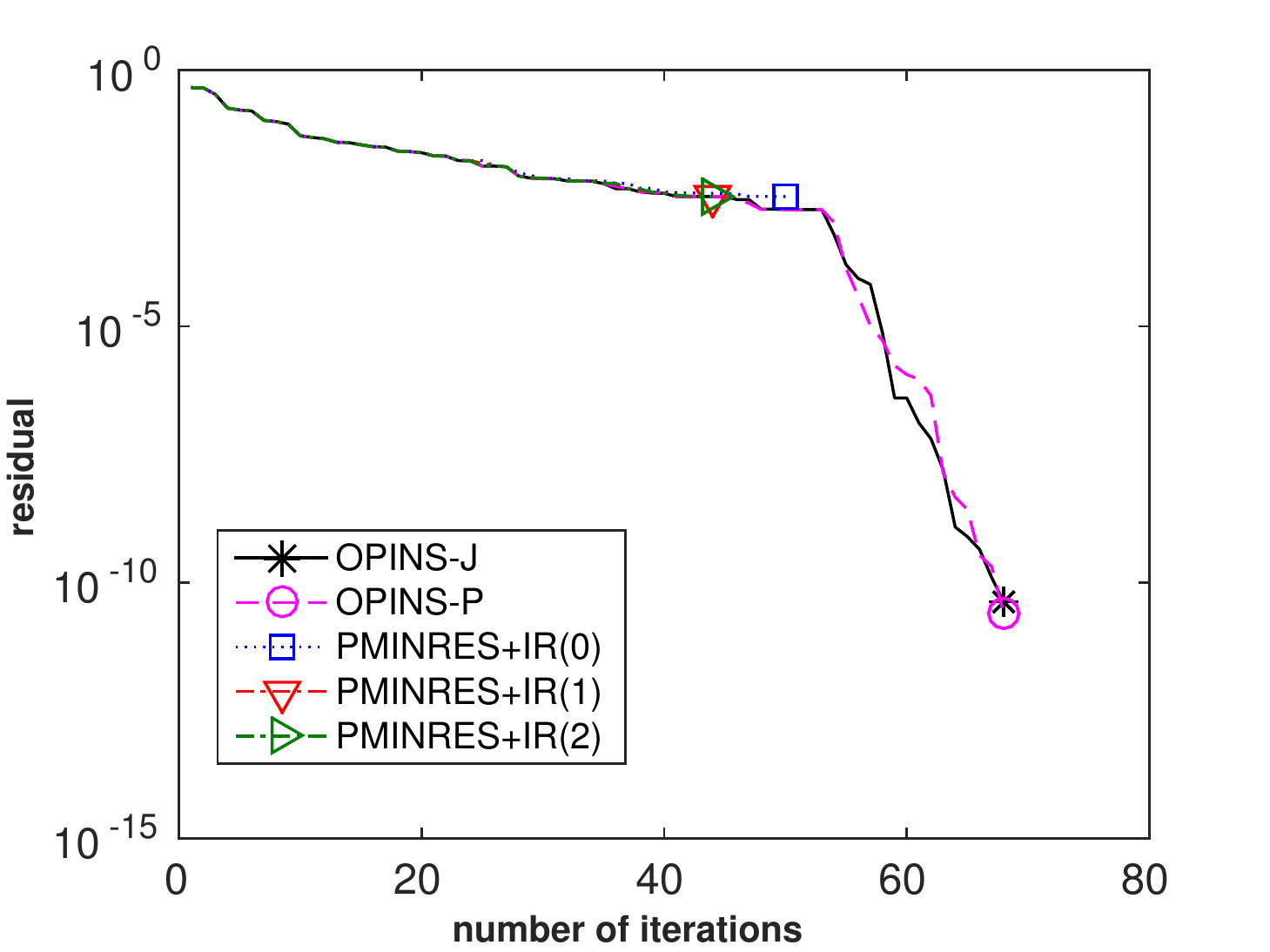}\\
{\small{}(a) Convergence history of relative residual in $\vec{x}$.}
\par\end{center}%
\end{minipage}%
\begin{minipage}[t]{0.5\columnwidth}%
\begin{center}
\includegraphics[width=1\columnwidth]{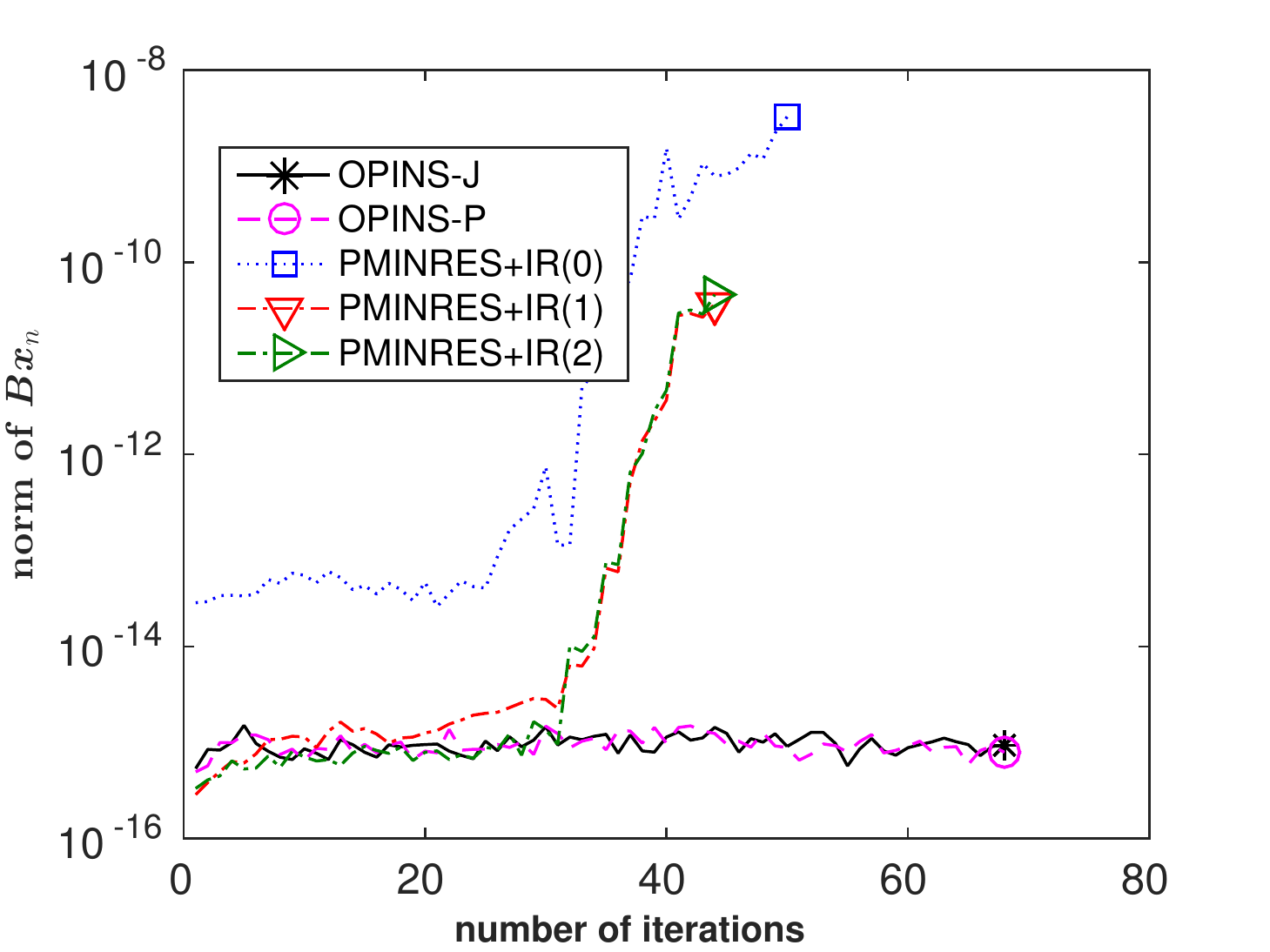}\\
{\small{}(b) History of $\Vert\vec{B}\vec{x}_{n}\Vert$ in OPINS and
PMINRES.}
\par\end{center}%
\end{minipage}
\par\end{centering}

\protect\caption{\label{fig:comparison3} Comparison of OPINS and projected MINRES
with/without iterative refinement for \textsf{can\_61}, for which
iterative refinement could not improve accuracy for projected MINRES
due to loss of orthogonality.}
\end{figure}

\subsection{\label{sub:Singular}Assessment for Symmetric Singular Systems}

When solving singular systems, no preconditioners are used to avoid
altering the minimum-norm solution. The methods for comparison are
CG, MINRES and GMRES with the same constraint preconditioner. For
GMRES, we calculate the error after the solver has converged.

In this subsection, we consider some test cases with singular systems.
The test problem \textsf{fracture}, comes from a finite element analysis
of nonlinear elasticity model with cracks. In this system, $\vec{A}$
represents the global stiffness matrix. The constraint matrix $\vec{B}$
is used to enforce various boundary conditions, such as sliding boundary
conditions and contact constraints. Due to the cracks opening along
edges, some pieces are isolated completely from the main body. This
introduces additional singularity to $\vec{A}$ such that $\mbox{null}\left(\vec{A}\right)\cap\mbox{null}\left(\vec{B}\right)\neq\left\{ 0\right\} $.
In addition, $\vec{B}$ contains redundant constraints, and it is
rank-deficient. We apply OPINS without preconditioners to solve the
problem. For PMINRES, PCG and PGMRES, $\vec{G}$ is set to be the
identity in the constraint preconditioner. Since $\vec{B}$ is rank-deficient,
an extra step of QR factorization is used to remove linearly dependent
rows.

\subsubsection{Convergence Comparison}

Figure~\ref{fig:fracture} shows the convergence of strategies. It
can be seen that OPINS converges monotonically to the desired accuracy
while the errors of PMINRES and PCG oscillate. In particular, the
error of PCG starts to grow after a certain number of iterations.
This shows that CG is not as robust as MINRES for solving singular
systems. In terms of the overall residual, OPINS has around $10^{-9}$
while the residuals of PCG and PMINRES stay around $1$. The reason
why PCG and PMINRES have large overall errors is that they do not
necessarily minimize the error of $[\vec{x},\vec{y}]$. To get a more
accurate $\vec{y}$, we need to solve the least-squares problem $\vec{B}^{T}\vec{y}=\vec{f}-\vec{A}\vec{x}$
once convergence in $\vec{x}$ is reached. On the other hand, the
residual of GMRES is about $10^{-6}$. However the converged solution
is incorrect, since $\Vert\vec{g}-\vec{B}\vec{x}\Vert$ is around
$1$. Overall OPINS is more stable and accurate than the other approaches.

\begin{figure}
\begin{centering}
\begin{minipage}[t]{0.5\columnwidth}%
\begin{center}
\includegraphics[width=1\columnwidth]{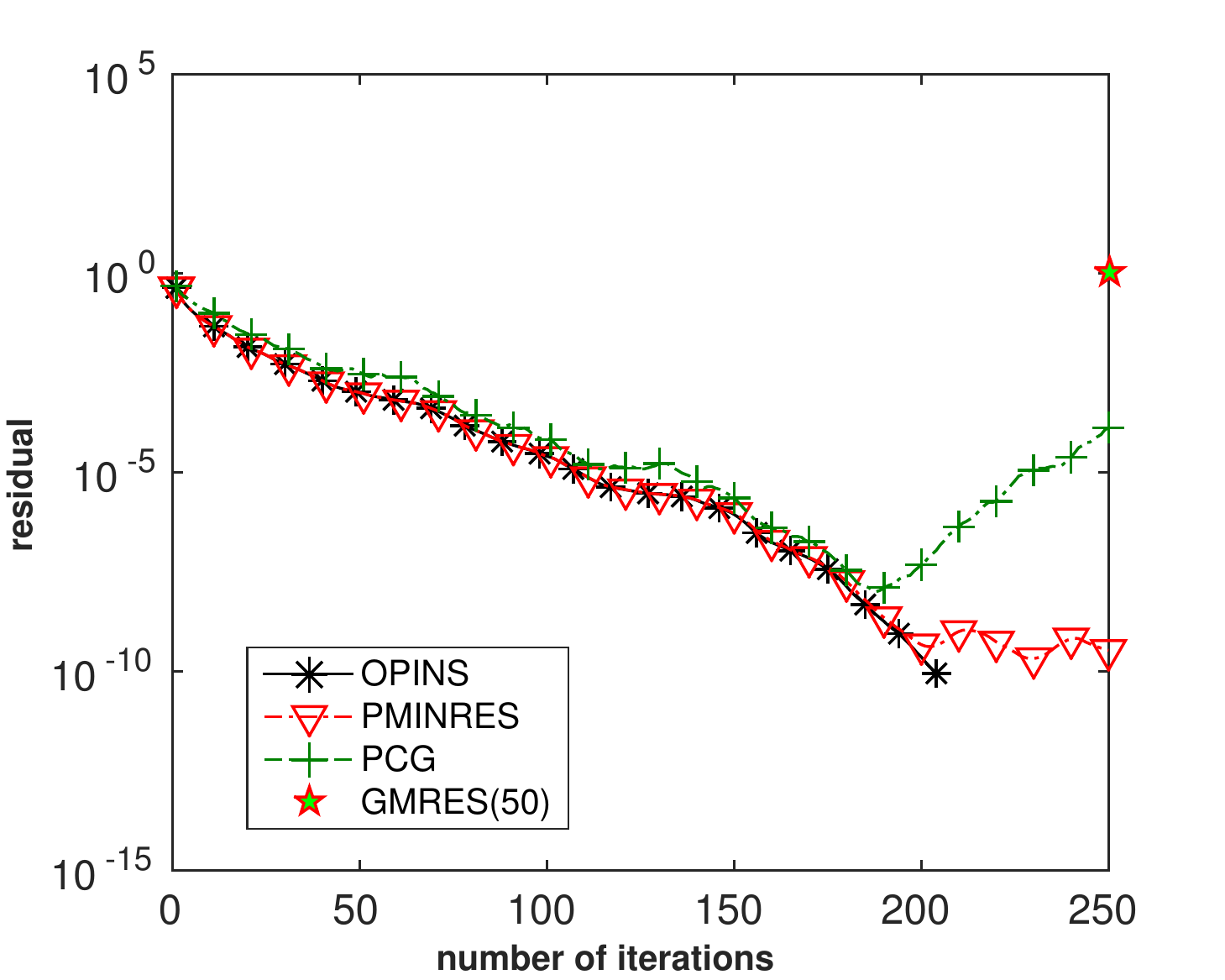}
\par\end{center}%
\end{minipage}
\par\end{centering}

\protect\caption{\label{fig:fracture} Convergence history of relative residual in
{\small{}$x$} for OPINS and Krylov methods for singular system from
\textsf{fracture}.}
\end{figure}

\subsubsection{Solution Norm}

To illustrate the minimum norm property of OPINS, we compare OPINS
with truncated SVD. The $\vec{x}$ and $\vec{y}$ components are measured
by 2-norm. The relative 2-norm of the residuals are provided as references.
In the fracture system, large Lamè parameters are used in calculating
element stresses. As a result, there is a $10^{10}$ gap between the
scaling of $\vec{A}$ and $\vec{B}$ which makes the system very ill-conditioned.
The conditioning will be improved if we scale $\vec{A}$ and $\vec{f}$
by $10^{-10}$. After the operation, $\vec{x}$ should remain unchanged
while $\vec{y}$ will be scaled by $10^{-10}$ accordingly. We apply
OPINS and truncated SVD (TSVD) to both scaled and unscaled systems,
where the tolerance for TSVD was $10^{-12}$. As shown in Table~\ref{tab:solution_norms},
OPINS is stable with and without scaling. Its $\vec{x}$ component
stays unchanged, and it has the same norm as the TSVD solution on
the scaled system. Without scaling, TSVD produced a wrong solution
that is far away from the constraint hyperplane. This is evident from
Table~\ref{tab:solution_errors}, where $\Vert\vec{g}-\vec{B}\vec{x}\Vert$
is very large for TSVD. In contrast, OPINS finds the minimum-norm
$\vec{x}$ independently of the scaling. Therefore, it is advantageous
to solve $\vec{x}$ and $\vec{y}$ separately over solving them together,
and it is not advisable to use any linear solver, including TSVD or
GMRES, as black box solvers for saddle-point problems.

In terms of the $\vec{y}$ component, the solution of OPINS has a
slightly larger norm than the TSVD solution for the scaled system.
This is due to the rank deficiency of $\vec{B}$. For the case where
$\vec{B}$ has full rank, we consider a random system. $\vec{A}$
is a $100\times100$ dense and semi-definite matrix with $\mbox{rank}=50$.
$\vec{B}$ is a randomly generated $20\times100$ dense matrix with
full rank. By construction, the system is singular with rank equal
to 90. Since $\vec{A}$ is semi-definite and $\vec{B}$ has full rank,
we can expect both $\vec{x}$ and $\vec{y}$ components of the OPINS
solution to have minimum norms, which is evident in Table~\ref{tab:solution_norms}.

\begin{table}
\protect\caption{\label{tab:solution_norms}Norms of solutions of OPINS versus truncated
SVD.}

\centering{}%
\begin{tabular}{|c|c|c|c|c|}
\hline 
\multirow{2}{*}{Test Problems} & \multicolumn{2}{c|}{OPINS} & \multicolumn{2}{c|}{TSVD}\tabularnewline
\cline{2-5} 
 & $\Vert\vec{x}\Vert$ & $\Vert\vec{y}\Vert$ & $\Vert\vec{x}\Vert$ & $\Vert\vec{y}\Vert$\tabularnewline
\hline 
\hline 
\textsf{fracture} & $1.53\times10^{-4}$ & $1.79\times10^{7}$ & $2.86\times10^{-5}$ & $4.25\times10^{-11}$\tabularnewline
\hline 
\textsf{scaled fracture} & $1.53\times10^{-4}$ & $1.79\times10^{-3}$ & $1.53\times10^{-4}$ & $1.54\times10^{-3}$\tabularnewline
\hline 
\textsf{random-s} & $5.56$ & $2.41$ & $5.56$ & $2.41$\tabularnewline
\hline 
\end{tabular}
\end{table}

\begin{table}
\protect\caption{\label{tab:solution_errors}Errors of solutions of OPINS versus truncated
SVD.}

\centering{}%
\begin{tabular}{|c|c|c|c|c|}
\hline 
\multirow{2}{*}{Test Problems} & \multicolumn{2}{c|}{OPINS} & \multicolumn{2}{c|}{TSVD}\tabularnewline
\cline{2-5} 
 & $\Vert\vec{g}-\vec{B}\vec{x}\Vert$ & $\Vert\vec{r}\Vert$ & $\Vert\vec{g}-\vec{B}\vec{x}\Vert$ & $\Vert\vec{r}\Vert$\tabularnewline
\hline 
\textsf{fracture} & $1.0\times10^{-17}$ & $1.2\times10^{-9}$ & $1.1$ & $1.9\times10^{-10}$\tabularnewline
\hline 
\textsf{scaled fracture} & $8.6\times10^{-18}$ & $5.6\times10^{-10}$ & $1.1\times10^{-14}$ & $1.3\times10^{-13}$\tabularnewline
\hline 
\textsf{random-s} & $3.7\times10^{-16}$ & $2.1\times10^{-13}$ & $2.3\times10^{-14}$ & $4.6\times10^{-15}$\tabularnewline
\hline 
\end{tabular}
\end{table}

\section{\label{sec:Conclusions}Conclusions}

In this paper, we introduced a new implicit null-space method, called
OPINS, for solving saddle-point systems, especially for applications
with relatively few constraints compared to the number of solution
variables. These systems may be nonsingular, or be singular but compatible
in terms of $\vec{f}$. Instead of finding the null space of the constraint
matrix explicitly, OPINS uses an orthonormal basis of its orthogonal
complementary subspace to reduce the saddle-point system to a singular
but compatible system. We showed that OPINS is equivalent to a null-space
method with an orthonormal basis for nonsingular systems. In addition,
it can solve singular systems and produce the minimum-norm solution,
which is desirable for many applications. Because of its use of orthogonal
projections, OPINS is more stable than other implicit null-space methods,
such as the projected Krylov methods. Despite its core equation is
singular, its nonzero eigenvalues have the same distribution as that
of the null-space method with an orthonormal basis, so an iterative
solver would converge at the same rate. We proposed effective preconditioners
for OPINS, based on the work for projected Krylov methods. For singular
saddle-point problems, one should only use left preconditioners that
do not alter the null space of the coefficient matrix, to ensure the
minimum-norm solution. A future research direction is to develop more
effective preconditioners for singular saddle-point systems. Another
direction is the parallelization of OPINS, which will require a parallel
rank-revealing QR factorization, such as that in \cite{davis2011algorithm},
as well as parallel KSP solvers, such as those available in PETSc
\cite{petsc-user-ref}.

\bibliographystyle{siam}
\bibliography{refs/refs,KKT}

\begin{thebibliography}{10}

\bibitem{adams2004algebraic}
{\sc M.~F. Adams}, {\em Algebraic multigrid methods for constrained linear
  systems with applications to contact problems in solid mechanics}, Numer.
  Linear Algebra Appl., 11 (2004), pp.~141--153.

\bibitem{petsc-user-ref}
{\sc S.~Balay, K.~Buschelman, V.~Eijkhout, W.~D. Gropp, D.~Kaushik, M.~G.
  Knepley, L.~C. McInnes, B.~F. Smith, and H.~Zhang}, {\em {PETSc} users
  manual}, Tech. Rep. ANL-95/11 - Revision 3.0.0, Argonne National Laboratory,
  2008.

\bibitem{bank1989class}
{\sc R.~E. Bank, B.~D. Welfert, and H.~Yserentant}, {\em A class of iterative
  methods for solving saddle point problems}, Numer. Math., 56 (1989),
  pp.~645--666.

\bibitem{BBC94Templates}
{\sc R.~Barrett, M.~Berry, T.~F. Chan, J.~Demmel, J.~Donato, J.~Dongarra,
  V.~Eijkhout, R.~Pozo, C.~Romine, and H.~V. der Vorst}, {\em Templates for the
  Solution of Linear Systems: Building Blocks for Iterative Methods, 2nd
  Edition}, SIAM, Philadelphia, PA, 1994.

\bibitem{bathe1975finite}
{\sc K.~J. Bathe, E.~Ramm, and E.~L. Wilson}, {\em Finite element formulations
  for large deformation dynamic analysis}, Int. J. Numer. Meth. Engrg., 9
  (1975), pp.~353--386.

\bibitem{ANU:298722}
{\sc M.~Benzi, G.~H. Golub, and J.~Liesen}, {\em Numerical solution of saddle
  point problems}, Acta Numer., 14 (2005), pp.~1--137.

\bibitem{Benzi08SPT}
{\sc M.~Benzi and A.~Wathen}, {\em Some preconditioning techniques for saddle
  point problems}, in Model Order Reduction: Theory, Research Aspects and
  Applications, W.~H.~A. Schilders, H.~A. van~der Vorst, and J.~Rommes, eds.,
  vol.~13 of Mathematics in Industry, Springer, 2008, pp.~195--211.

\bibitem{Boisvert:1997:MMW:265834.265854}
{\sc R.~F. Boisvert, R.~Pozo, K.~Remington, R.~F. Barrett, and J.~J. Dongarra},
  {\em Matrix market: A web resource for test matrix collections}, in
  Proceedings of the IFIP TC2/WG2.5 Working Conference on Quality of Numerical
  Software: Assessment and Enhancement, London, UK, 1997, Chapman \& Hall,
  Ltd., pp.~125--137.

\bibitem{bramble1988preconditioning}
{\sc J.~H. Bramble and J.~E. Pasciak}, {\em A preconditioning technique for
  indefinite systems resulting from mixed approximations of elliptic problems},
  Math. Comput., 50 (1988), pp.~1--17.

\bibitem{chan1987rank}
{\sc T.~F. Chan}, {\em Rank revealing {QR} factorizations}, Linear Algebra
  Appl., 88 (1987), pp.~67--82.

\bibitem{choi2011minres}
{\sc S.-C.~T. Choi, C.~C. Paige, and M.~A. Saunders}, {\em {MINRES-QLP}: A
  {Krylov} subspace method for indefinite or singular symmetric systems}, SIAM
  J. Sci. Comput., 33 (2011), pp.~1810--1836.

\bibitem{davis2011algorithm}
{\sc T.~A. Davis}, {\em Algorithm 915, {SuiteSparseQR}: Multifrontal
  multithreaded rank-revealing sparse {QR} factorization}, ACM Trans. Math.
  Softw., 38 (2011), p.~8.

\bibitem{de2005block}
{\sc E.~de~Sturler and J.~Liesen}, {\em Block-diagonal and constraint
  preconditioners for nonsymmetric indefinite linear systems. {Part I}:
  Theory}, SIAM J. Sci. Comput., 26 (2005), pp.~1598--1619.

\bibitem{Fong11LSMR}
{\sc D.~Fong and M.~M. Saunders}, {\em {LSMR}: An iterative algorithm for
  sparse least-squares problems}, SIAM J. Sci. Comput., 33 (2011),
  pp.~2950--2971.

\bibitem{fries2011hanging}
{\sc T.-P. Fries, A.~Byfut, A.~Alizada, K.~W. Cheng, and A.~Schr{\"o}der}, {\em
  Hanging nodes and {XFEM}}, Int. J. Numer. Meth. Engrg., 86 (2011),
  pp.~404--430.

\bibitem{gilbert1987computing}
{\sc J.~R. Gilbert and M.~T. Heath}, {\em Computing a sparse basis for the null
  space}, SIAM J. Alg. Disc. Meth., 8 (1987), pp.~446--459.

\bibitem{Golub13MC}
{\sc G.~H. Golub and C.~F. {Van Loan}}, {\em Matrix Computations}, Johns
  Hopkins, 4th~ed., 2013.

\bibitem{gould2001solution}
{\sc N.~I.~M. Gould, M.~E. Hribar, and J.~Nocedal}, {\em On the solution of
  equality constrained quadratic programming problems arising in optimization},
  SIAM J. Sci. Comput., 23 (2001), pp.~1376--1395.

\bibitem{gould2014projected}
{\sc N.~I.~M. Gould, D.~Orban, and T.~Rees}, {\em Projected krylov methods for
  saddle-point systems}, SIAM J. Matrix Anal. \& Appl., 35 (2014),
  pp.~1329--1343.

\bibitem{greenbaum1994max}
{\sc A.~Greenbaum and L.~Gurvits}, {\em Max-min properties of matrix factor
  norms}, SIAM J. Sci. Comput., 15 (1994), pp.~348--358.

\bibitem{Hthesis}
{\sc J.~Huang}, {\em Constrained Variational Analysis Integrating Vertical and
  Temporal Correlations}, PhD thesis, Applied Math and Statistics, Stony Brook
  University, 2012.

\bibitem{keller2000constraint}
{\sc C.~Keller, N.~I.~M. Gould, and A.~J. Wathen}, {\em Constraint
  preconditioning for indefinite linear systems}, SIAM J. Matrix Anal. \&
  Appl., 21 (2000), pp.~1300--1317.

\bibitem{maros1999repository}
{\sc I.~Maros and C.~M{\'e}sz{\'a}ros}, {\em A repository of convex quadratic
  programming problems}, Optim. Method. Softw., 11 (1999), pp.~671--681.

\bibitem{nocedal2006numerical}
{\sc J.~Nocedal and S.~Wright}, {\em Numerical Optimization}, Springer Science
  \& Business Media, 2006.

\bibitem{paige1975solution}
{\sc C.~C. Paige and M.~A. Saunders}, {\em Solution of sparse indefinite
  systems of linear equations}, SIAM J. Numer. Anal., 12 (1975), pp.~617--629.

\bibitem{Paige92LSQR}
{\sc C.~C. Paige and M.~A. Saunders}, {\em {LSQR}: An algorithm for sparse
  linear equations and sparse least squares}, ACM Trans. Math. Softw., 8
  (1982), pp.~43--71.

\bibitem{Reichel2005BGS}
{\sc L.~Reichel and Q.~Ye}, {\em Breakdown-free {GMRES} for singular systems},
  SIAM J. Matrix Anal. Appl., 26 (2005), pp.~1001--1021.

\bibitem{Saad03IMS}
{\sc Y.~Saad}, {\em Iterative methods for sparse linear systems}, SIAM,
  2nd~ed., 2003.

\bibitem{Saad86GMRES}
{\sc Y.~Saad and M.~Schultz}, {\em {GMRES}: A generalized minimal residual
  algorithm for solving nonsymmetric linear systems}, SIAM Journal on
  Scientific and Statistical Computing, 7 (1986), pp.~856--869.

\bibitem{schoberl2007symmetric}
{\sc J.~Sch{\"o}berl and W.~Zulehner}, {\em Symmetric indefinite
  preconditioners for saddle point problems with applications to
  {PDE}-constrained optimization problems}, SIAM J. Matrix Anal. \& Appl., 29
  (2007), pp.~752--773.

\bibitem{tuma2002note}
{\sc M.~Tuma}, {\em A note on the {$LDL^T$} decomposition of matrices from
  saddle-point problems}, SIAM J. Matrix Anal. \& Appl., 23 (2002),
  pp.~903--915.

\bibitem{vavasis1994stable}
{\sc S.~A. Vavasis}, {\em Stable numerical algorithms for equilibrium systems},
  SIAM J. Matrix Anal. \& Appl., 15 (1994), pp.~1108--1131.

\end{thebibliography}

\end{document}